%% file: MPR20120304.tex
\documentclass[11pt]{amsart}

\usepackage[osf]{mathpazo}

\input LaTeXdefs.tex

\usepackage{latexsym}
\usepackage{amsfonts, amssymb, amsmath}

\usepackage{color}

\newcommand{\abs}[1]{\lvert#1\rvert}

\renewcommand{\le}{\leqslant}
\renewcommand{\ge}{\geqslant}
\renewcommand{\mid}{\::\:}
\newcommand{\codim}{\mathrm{codim}}
\newcommand{\term}[1]{{\textit{\textbf{#1}}}}

\begin{document}

\title[On Almost-Invariant Subspaces and Approximate Commutation]
      {On Almost-Invariant Subspaces and Approximate Commutation}

\thanks{${}^1$ Research supported in part by NSERC (Canada)}
\thanks{\today}

\author[L.W. Marcoux]{{Laurent W.~Marcoux${}^1$}}
\author[A.~I.~Popov]{Alexey I. Popov${}^1$}
\author[H.~Radjavi]{Heydar Radjavi${}^1$}

\address
	{Department of Pure Mathematics\\
	University of Waterloo\\
	Waterloo, Ontario \\
	Canada  \ \ \ N2L 3G1}
\email{LWMarcoux@math.uwaterloo.ca}
\email{a4popov@uwaterloo.ca}
\email{hradjavi@uwaterloo.ca}

\begin{abstract}
A closed subspace of a Banach space $\cX$ is almost-invariant for a collection $\cS$ of bounded linear operators on $\cX$ if for each $T \in \cS$ there exists a finite-dimensional subspace $\cF_T$ of $\cX$ such that $T \cY \subseteq \cY + \cF_T$.   In this paper, we study the existence of almost-invariant subspaces of infinite dimension and codimension  for various classes of Banach and Hilbert space operators.   We also examine the structure of operators which admit a maximal commuting family of almost-invariant subspaces. In particular, we prove that if $T$ is an operator on a separable Hilbert space and if $TP-PT$ has finite rank for all projections $P$ in a given maximal abelian self-adjoint algebra $\fM$ then $T=M+F$ where $M\in\fM$ and $F$ is of finite rank.
\end{abstract}

\maketitle
\markboth{\textsc{  }}{\textsc{}}


\section{Introduction}
One of the best-known problems in Operator Theory concerns the search for non-trivial, closed, invariant subspaces for an operator (or class of operators) acting on an infinite-dimensional, separable Banach space $\cX$.  The existence or non-existence of such spaces has been proven for large classes of operators. Examples of operators without invariant subspaces have been found by Enflo~\cite{Enfl76,Enfl87} and Read~\cite{R84}. 
Moreover, Read constructed a number of operators without invariant subspaces. Among them are an operator acting on $\ell_1$~\cite{R85}, a quasinilpotent operator~\cite{R97}, a strictly singular operator~\cite{R99}, or an operator without invariant closed sets~\cite{R88}. The most recent construction is published by Sirotkin~\cite{Sir10}. The question of the existence of invariant subspaces for an arbitrary bounded linear operator acting on a reflexive Banach space, and in particular on a Hilbert space,  remains open.  This is the  Invariant Subspace Problem. We refer the reader to~\cite[Section~10]{AA02} for a brief review of this topic. Another source is the monograph~\cite{RR03}.

The problem which we shall examine in this paper is very closely related, but not equivalent, to this problem.  For the purposes of this paper, all subspaces of a Banach space are assumed to be closed in the norm topology.  Given a Banach space $\cX$ and a bounded linear operator $T$ on $\cX$, we shall say that a  subspace $\cY$ of $\cX$ is an \term{almost-invariant subspace} for $T$ if there exists a finite-dimensional subspace $\cM$ of $\cX$ so that $T \cY \subseteq \cY + \cM$. While the finite-dimensional space $\cM$ appearing in this expression is not unique, nevertheless, if the subspace $\cY$ is almost-invariant for $T$, then the minimum dimension of a finite-dimensional subspace $\cM$ for which $T \cY \subseteq \cY + \cM$ is well-defined, and is referred to as the \term{defect} of the subspace $\cY$ for $T$.  Clearly, if $\cY$ is finite-dimensional or finite-codimensional then it is almost-invariant under every operator. This motivates the notion of a \term{half-space}, that is, a subspace of $\cX$ such that  $\dim\, \cY = \dim\, (\cX/\cY) = \infty$. One may then ask whether every operator $T$ on a Banach space $\cX$ has an almost-invariant half-space.

These notions appeared in the papers~\cite{APTT} and~\cite{Popov2010}. In~\cite{APTT}, it was shown that almost-invariant half-spaces exist for a class of operators which includes quasinilpotent weighted shift operators acting on Banach spaces with bases. Among other results, it was shown in~\cite{APTT} that a closed subspace $\cY$ of $\cX$ is almost-invariant for $T$ if and only if there exists a finite-rank perturbation $F$ of $T$ so that $\cY$ is invariant for the operator $T+F$. Also, if $T$ has an almost-invariant half-space then so does $T^*$. The main thrust of~\cite{Popov2010} is the analysis of common almost-invariant half-spaces for algebras of operators. It was shown that if an algebra $\cA$ of operators on a Banach space $\cX$ is norm-closed then the dimensions of the defects corresponding to each operator in $\cA$ are uniformly bounded. Another result of~\cite{Popov2010} is that if a norm-closed algebra $\cA$ is finitely-generated and commutative then the existence of a common almost-invariant half-space for $\cA$ implies the existence of an invariant half-space for $\cA$. Also, it was shown that for a norm-closed algebra $\cA$, the almost-invariant half-spaces of $\cA$ and those of the closure of $\cA$ in the weak operator topology, $\overline{\cA}^{WOT}$, are the same. This is no longer true if $\cA$ is not closed in the norm topology.

The main result of~\cite{APTT} is motivated by the study of almost-invariant half-spaces of Donoghue operators. Recall that a (backward) weighted shift $D$ on $\ell_2$ is called a \term{Donoghue operator} if its weights $(w_i)$ are all non-zero and $\abs{w_i}\downarrow 0$. See, e.g., \cite[Section~4.4]{RR03} for a discussion of Donoghue operators. The following property of Donoghue operators is very interesting from our point of view (see~\cite{D57}, \cite{N65} and~\cite{Y85}):


\begin{thm}
The invariant subspaces of a Donoghue operator are all of the form $\mathrm{span}\,\{e_k\mid k=1,\dots,n\}$, where $(e_n)$ stands for the usual orthonormal basis of $\ell_2$. In particular, Donoghue operators do not have invariant half-spaces.
\end{thm}

The following result from~\cite{APTT} implies that Donoghue operators have almost-invariant half-spaces. 

\begin{thm}\label{donoghue-aihs}{ \rm\cite[Theorem~3.2]{APTT}}
Let $T$ be an operator on a Banach space $\cX$. Suppose that the following three conditions are satisfied.
\begin{enumerate}
\item\label{APTT-condition-1} $T$ has no eigenvalues;
\item\label{APTT-condition-2} the unbounded component of the resolvent set $\rho(T)$ contains $\{z\in\mathbb C\mid 0<\abs{z}<\varepsilon\}$ for some $\varepsilon>0$;
\item\label{APTT-condition-3} there exists $e\in\cX$ such that for each $k\in\mathbb N$, the vector $T^ke$ does not belong to the closed span of the set of vectors $\{T^ie\mid i\ne k\}$.
\end{enumerate}
Then $T$ has an almost-invariant half-space, $\cY$.
\end{thm}

\begin{rem}\label{redundant-condition}
Condition~\eqref{APTT-condition-1} in Theorem~\ref{donoghue-aihs} is, in fact, redundant. To justify this claim, we need to outline the proof of this theorem. The half-space $\cY\subseteq\cX$ almost-invariant for $T$ is constructed as a closed span of the form $\overline{\mathrm{span}\,}\{h(\lambda,e)\mid\lambda\in\Lambda\}$ where $\Lambda$ is a certain infinite subset of $\rho(T)$ and $h(\lambda,e)=(\lambda I-T)^{-1}(e)$, $\lambda\in\Lambda$.  (We warn the reader that our definition of $\Lambda$ and $h(\lambda,e)$ differs slightly from the definition presented in~\cite{APTT}.)  Condition~\eqref{APTT-condition-1} is used in the following two steps:
\begin{enumerate}
	\item[(1)] the set $\{h(\lambda,e)\mid\lambda\in\Lambda\}$ is linearly independent, so $\cY$ is of infinite dimension;
	\item[(2)] a sequence of functionals $(f_k)_{k=1}^\infty$ annihilating $\cY$ is defined by assigning values to $f_k(T^{i}e)$ ($i\ge 0$) which requires  the sequence $(T^{i}e)_{i\ge 0}$ to be linearly independent.
\end{enumerate}
We claim that both~(1) and~(2) follow from the condition
\begin{enumerate}
	\item[(i')] $p(T)e\ne 0$ for any non-zero polynomial $p$,
\end{enumerate}
which, clearly, follows from  condition~\eqref{APTT-condition-3} of Theorem~\ref{donoghue-aihs}.

Indeed, it is evident that the linear independence of $(T^{i} e)_{i \ge 0}$ follows  from~(i'). Let us show that~(i') implies~(1). Assuming~(i'), we will show that the set $\{h(\lambda,e)\mid\lambda\in\Lambda\}$ is linearly independent for any choice of~$\Lambda$. Suppose that for some non-zero scalars $a_1,\dots,a_n$ and distinct $\lambda_1,\dots,\lambda_n$ in~$\Lambda$, we have
$$
a_1h(\lambda_1,e)+ a_2h(\lambda_2,e)+\dots+a_nh(\lambda_n,e)=0
$$
That is, 
$$
\sum_{i=1}^na_i(\lambda_i I-T)^{-1}(e)=0.
$$
Multiplying both sides of this equation by $(\lambda_1  I-T) (\lambda_2  I-T)\cdots(\lambda_n I-T)$, we obtain: $p(T)(e)=0$ for some (non-zero, by elementary algebra) polynomial~$p$. This contradicts~(i').
\end{rem}
\bigskip

There are other closely related questions which have already been examined.   For example, in the Hilbert space setting, Brown and Pearcy~\cite{BP1971} showed that if $T$ is a bounded linear operator on an infinite-dimensional Hilbert space $\cH$, then there exists an orthogonal projection $P$ of $\cH$ onto a half-space $\cL$ of $\cH$ so that $K= (I-P) T P$ is compact, i.e. so that $\cL$ is invariant for the operator $T - K$.   In fact, the work of Fillmore, Stampfli and Williams~\cite{FSW72} shows that one can do this by considering points in the left-essential spectrum of $T$.   A direct consequence of Voiculescu's non-commutative Weyl-von Neumann Theorem~\cite{Voi1976} is that given any bounded linear operator $T$ acting on an infinite-dimensional Hilbert space $\cH$, there exists a compact operator $K$ and a closed half-space $\cL$ of $\cH$ which is \term{reducing} for $T-K$;  that is, $\cL$ is invariant for both $T-K$ and $(T-K)^*$.

In Section~\ref{sec2} of this paper, we obtain almost-invariant subspaces for some classes of operators. First, we show that every quasinilpotent, triangularizable injective operator has an almost-invariant half-space (see Section~\ref{sec2} for the definition of a triangularizable operator). Next, we use this result to show amongst other things that every polynomially compact operator on a reflexive Banach space has an almost-invariant half-space.  Observe that the class of Donoghue operators which served as motivation for Theorem~\ref{donoghue-aihs} consists of compact operators. Finally, we study a special class of triangularizable operators: bitriangular operators. We show that every bitriangular operator is either of form $\lambda I+F$, for some scalar $\lambda$ and a finite-rank operator $F$, or it has a hyperinvariant half-space. In particular, every bitriangular operator has an invariant half-space.

Section~\ref{sec3} is concerned with the study of common almost-invariant half-spaces for algebras of operators. We extend a result from~\cite{Popov2010} by showing that if a norm-closed algebra $\cA$ of operators on a Banach space $\cX$ has an almost-invariant half-space which is complemented in $\cX$, then $\cA$ has a common invariant half-space. In particular, if $\cX$ is a Hilbert space then the existence of an almost-invariant half-space for a norm-closed algebra always implies the existence of an invariant half-space for that algebra. That is, if $P$ is a projection onto a half-space in $\cH$ such that $TP-PTP$ is of finite rank for all $T$ in the algebra then the algebra has an invariant half-space. We would like to bring to the reader's attention that the corresponding statement in which ``finite rank'' is replaced with ``compact'' is not true, as the following example shows:
\begin{eg} Let $\cD$ be the $C^*$-algebra of all diagonal (with respect to a fixed orthonormal basis) operators in $\bofh$. Let $\cA=\{D+K\mid D\in\cD,\ K\in\cK(\hilb)\}=\cD+\cK(\hilb)$. It is well known (see, e.g., \cite[Theorem 3.1.7]{Murphy}) that $\cA$ is a norm closed subalgebra of $\bofh$. Clearly, $TP-PTP\in\cK(\hilb)$ for all $T\in\cA$ and every projection $P\in\cD$. On the other hand, $\cA$ contains all compact operators in $\bofh$, hense it is dense in $\bofh$ in the weak operator topology. In particular, $\cA$ has no invariant subspaces.
\end{eg}

\smallskip

In 1972, Johnson and Parrot~\cite{JP1972} showed that if $\cW$ is an abelian von Neumann algebra acting on a Hilbert space $\cH$, and if $T$ is a bounded operator on $\hilb$ for which $T W - W T$ is compact for all $W \in \cW$, then there exists an operator $V$ in the commutant $\cW^\prime := \{ X \in \bofh: X W = W X \mbox{ for all } W \in \cW\}$ of $\cW$ and $K$ compact so that $T = V + K$.  In particular, if $\cW$ is a maximal abelian self-adjoint subalgebra (i.e. a masa) in $\bofh$, then $V \in \cW$.  The condition that $T W - W T$ be compact for all $W \in \cW$ is readily seen to be equivalent to the condition that $T P - P T$ be compact for all projections in $\cW$, since the linear span of the projections is norm-dense in $\cW$ by the Spectral Theorem for normal operators.
In Section~\ref{sec4} of the present work we replace this condition with the related condition that $T P - P T$ be finite-rank, or equivalently, that the range of $P$ be an almost-invariant subspace for $T$.  More specifically, we study operators with the following property: there exists a   masa  $\cD$ in $\cB(\hilb)$ such that for every projection $P\in\cD$, the space $P \hilb$ is almost-invariant under the given operator.  We show that every such operator can be written in the form $D+F$ where $D$ is in the masa and $F$ is of finite rank.   Unlike in the setting of the Johnson and Parrot result, the fact that the operator $T$ has finite-rank commutators $T M - M T$ with every element $M$ in the masa (and not just with the projections $P$ in the masa) is part of the conclusion of our result, not of the hypothesis.  In the case of a continuous masa, this decomposition is unique. We also study norm-closed algebras of operators leaving all spaces of the form $P \hilb$ (for $P$ in a masa $\cD$) almost-invariant. We show that every such algebra is contained in $\cD+\cF_k(\hilb)$ for some $k\in\bbN$, where $\cF_k(\hilb)$ denotes the set of operator on $\hilb$ of rank at most~$k$. We use these results to show that if an operator on $\hilb$ leaves every half-space of~$\hilb$ almost-invariant then it can be written as $\lambda I+F$ for some scalar $\lambda$ and  finite-rank operator~$F$.

In the last section of the paper, we introduce the notion of an almost-reducing subspace. Recall that a space $\cY\subseteq\hilb$ is reducing for a  collection $\cS$ of operators  in $\cB(\hilb)$ if $\cY$ and $\cY^\perp$ are both invariant for each $T \in \cS$.  We say that  $\cY$ is \term{almost-reducing} for $\cS$ if $\cY$ and $\cY^\perp$ are each almost-invariant for $\cS$.  We show that an analogue of our result from Section~\ref{sec3} for almost-reducing subspaces does not hold. There exist norm-closed algebras with many common almost-reducing subspaces that do not have common reducing subspaces. In fact, one can find a singly generated norm-closed algebra with plenty of almost-reducing half-spaces whose generator does not have reducing subspaces.

Throughout this work, $\cX$, $\cY$ and $\cZ$ will denote complex Banach spaces, while $\cH$, $\cK$ and $\cL$ will be reserved for (complex) Hilbert spaces.   
We shall consistently use the following notation:   for $\cX$ a Banach space, $\cB(\cX), \cK(\cX), \cF(\cX)$ and $\cF_k(\cX)$ are the algebras, respectively, of all bounded operators, the compact operators, the finite-rank operators, and those of finite-rank less than or equal to $k$ for a given $k \in \bbN$.   For an operator $T$, $\sigma(T)$, $\sigma_p(T)$ and $\mathrm{rank}\, T$ are, respectively,  the spectrum of $T$, the point spectrum of $T$ (i.e. the set of eigenvalues of $T$), and the rank of $T$.   For a subset $\cS$ of $\cX$, $\overline{\cS}$ is the closure of $\cS$ in the norm topology, and if $\cS = \{ x_n \}_{n=1}^\infty$ is a countable subset of $\cS$, then $[ x_n ]_n$ is the closed linear span of $\cS$.   If $ \{ \cY_i : i \in I\}$ is an arbitrary collection of subspaces of $\cX$, then $\bigvee_{i \in I} \cY_i$ is their closed linear span.
We call an operator $P \in \cB(\hilb)$ a projection if $P=P^2=P^*$. An operator (on a Banach space) that only satisfies $E=E^2$ will be referred to as idempotent.  Finally, if $\cY \subseteq \cX$ is a subspace, then $\mathrm{codim}_\cX\, \cY := \mathrm{dim}\, \cX/\cY$.   We shall also use the notation $\mathrm{codim}\, \cY$ if the space $\cX$ is understood.

The following (obvious) lemma will be used many times throughout the paper without further reference.

\begin{lem}
Let $\cX$ be a Banach space, $\cY$ a complemented subspace of $\cX$ and $E_\cY:\cX\to\cX$ be an idempotent with range~$\cY$. Then $\cY$ is almost-invariant under operator $T$ on $\cX$ if and only if the operator $(I-E_\cY)TE_\cY$ is of finite rank. Moreover, the defect of $\cY$ for $T$ is equal to $\mathrm{rank}\,\bigl((I-E_\cY)TE_\cY\bigr)$.
\end{lem}


\section{Existence of almost-invariant half-spaces for single operators}\label{sec2}

In this section we will establish the existence of almost-invariant half-spaces for several classes of operators. 
The following lemma proves the simple fact that half-spaces exist in all Banach spaces.

\begin{lem}\label{hs-exist}
Every infinite-dimensional Banach space contains a half-space.
\end{lem}

\begin{proof}
Let $(x_n)$ be a basic sequence in $\cX$ (see, e.g., \cite[Theorem~1.a.5]{LT77}). Put $\cY=[x_{2n}]_{n=1}^\infty$. We claim that $\cY$ is a half-space. 

It is clear that $\dim \, \cY=\infty$. To see that $\codim\, \cY=\infty$, observe first that $(x_{2n})$ is also a basic sequence, so that every member of $\cY$ can be written as $\sum_{i=n}^\infty a_nx_{2n}$. Since $(x_n)$ is basic, $\cY$ contains no nonzero elements of form $\sum_{n=1}^m b_nx_{2n-1}$, where $m\in\mathbb N$. It follows that $\codim_{[x_n]}\cY=\infty$. Hence, $\codim_\cX \cY=\infty$.
\end{proof}


We remind the reader that an operator $T$ on a Banach space $\cX$ is called \term{triangularizable} if there is a chain $\mathcal C$ of subspaces of $\cX$ that is maximal with respect to inclusion and has the property that each member $\cY$ of $\mathcal C$ is $T$-invariant (see, e.g., \cite[Definition~7.1.1]{RR00}). We point out that the maximality of the chain $\cC$ implies that $\cC$ is complete; that is, it is closed under arbitrary intersections, as well as under closed spans of its members.

Our first goal is to show that all injective, triangularizable quasinilpotent operators admit an almost-invariant half-space. We will start with two simple lemmas.


\begin{lem}\label{minimal-sequences}
Let $(x_n)$ be a sequence in a Banach space $\cX$. Then $x_n\not\in[x_k]_{k\ne n}$ holds for all $n\in\mathbb N$ if and only if $x_n\not\in[x_k]_{k>n}$ holds for all $n\in\mathbb N$.
\end{lem}
\begin{proof}
The ``only if'' part of the lemma is trivial. Let us establish the ``if'' part. Let $(x_n)$ be such that $x_n\not\in[x_k]_{k>n}$ holds for all $n\in\mathbb N$. Observe that the space $\mathrm{span}\{x_k\}_{k<n}+[x_k]_{k>n}$ is dense in $[x_k]_{k\ne n}$. Since the sum of a finite-dimensional space and a closed space is again closed, this shows that $[x_k]_{k\ne n}=\mathrm{span}\{x_k\}_{k<n}+[x_k]_{k>n}$.

Suppose that $x_n\in[x_k]_{k\ne n}$ for some $n\in\mathbb N$. Then we can write $x_n=a_1x_1+\dots+a_{n-1}x_{n-1}+u$ where $u\in[x_k]_{k> n}$. Put $a_n=-1$ and let $i$ be the smallest number such that $a_i\ne 0$. Then $x_i=-\frac{a_{i+1}}{a_i}x_{i+1}-\dots-\frac{a_{n}}{a_i}x_{n}-\frac{1}{a_i}u$. This shows that $x_i\in[x_k]_{k>i}$, a contradiction.
\end{proof}


\begin{lem}\label{qn-quotient}
Let $T$ be a quasinilpotent operator on a Banach space $\cX$.   If $\cZ \subseteq \cY$ are $T$-invariant subspaces of $\cX$ and $\dim(\cY/\cZ)=1$ then $T\cY \subseteq \cZ$.
\end{lem}
\begin{proof}
Write $\cY$ as a direct sum $\cY=\cZ\oplus[y]$ for some $y\in\cY\setminus\cZ$ with $\left\|y\right\|=1$. Let $E:\cY\to\cY$ be the idempotent with range $[y]$ and null space~$\cZ$. If $T\cY\not\subseteq\cZ$ then there exists a non-zero $\alpha\in\mathbb C$ such that $Ty=\alpha y+z$ for some $z\in\cZ$. Then, since $T\cZ\subseteq\cZ$, we obtain $E T^n y=\alpha^n y$, for all $n\in\mathbb N$. However, this implies that $\left\|T^n\right\|\ge\frac{1}{\left\|E\right\|}\left\|ET^n\right\|\ge\frac{1}{\left\|E\right\|}\abs{\alpha}^n$, so that the spectral radius $r(T)$ of $T$ satisfies the inequality $r(T)\ge\abs{\alpha}>0$, a contradiction.
\end{proof}


\begin{thm}\label{qn-aihs} 
Every triangularizable quasinilpotent injective operator $T$ on a Banach space $\cX$ has an almost-invariant half-space.
\end{thm}
\begin{proof}
Let $\mathcal C$ be a triangularizing chain of invariant subspaces for $T$.  Clearly, we may assume without loss of generality that no member of $\mathcal C$ is a half-space. Also, it follows from the injectivity of $T$ that $\mathcal C$ has no elements having finite dimension, as the restriction of $T$ to such a subspace would be nilpotent, thereby contradicting the injectivity of $T$.  Therefore, every member of $\mathcal C$ is of finite codimension in~$\cX$.

Following \cite[Definition~7.1.6]{RR00}, for each $\cY \in\mathcal C$, we define $\cY_{-}=\bigvee\{\cZ \in\mathcal C\mid\cZ \subseteq\cY,\cZ \ne\cY \}$. It follows from maximality of $\mathcal C$ that for each $\cY \in\mathcal C$, either $\cY_{-}=\cY $ or $\dim(\cY /\cY_{-})=1$ (see \cite[Theorem~7.1.9(iii)]{RR00}). However, the condition $\cY_{-}=\cY$ implies that $\mathcal C$ contains members having infinite codimension in $\cX$. So, $\dim(\cY/\cY_{-})=1$ for all $\cY \in\mathcal C$.

Let $\cY_1$ be an arbitrary nonzero element of $\mathcal C$. Pick a non-zero vector $e_1\in \cY_1\setminus \cY_{1-}$. Denote $e_2=Te_1$. By Lemma~\ref{qn-quotient}, $e_2\in \cY_{1-}$. Also, since $\ker T=\{0\}$, $e_2\ne 0$.

Let $\cY_2=\bigcap\{\cY \in\mathcal C\mid e_2\in \cY \}$. Since $\mathcal C$ is a complete chain, $\cY_2\in\mathcal C$. Clearly, $\cY_2\ne\{0\}$ and $\cY_2\subsetneq \cY_1$. Also, $\cY_{2-}\ne \cY_2$. In particular, $e_2\in \cY_2\setminus \cY_{2-}$.

Denote $e_3=Te_2$. Again, $e_3\ne 0$ and $e_3\in \cY_{2-}$. Put $\cY_3=\bigcap\{ \cY \in\mathcal C\mid e_3\in \cY\}$. Then $\cY_3\in\mathcal C$ is such that $\cY _3\ne\{0\}$, $\cY_3\subsetneq \cY_2$, and $e_3\in \cY_3\setminus \cY_{3-}$.

Continuing inductively, we obtain a strictly decreasing chain of subspaces $\cY_1\supsetneq \cY_2\supsetneq \cY_3\supsetneq \cY_4\supsetneq\dots$ in $\cC$ such that if $e_k=T^{k-1}e_1$ then $e_k\in \cY_k\setminus \cY_{k+1}$. The sequence $(e_k)$, in particular, satisfies the condition that $e_i\not\in[e_j]_{j>i}$ for all $i\in\mathbb N$. By Lemma~\ref{minimal-sequences}, $e_i\not\in[e_j]_{j\ne i}$ for all $i\in\mathbb N$. It follows from Theorem~\ref{donoghue-aihs} that $T$ has an almost-invariant half-space.
\end{proof}


If the underlying Banach space in Theorem~\ref{qn-aihs} is reflexive then we can drop the injectivity assumption from the hypothesis of the theorem.

\begin{thm}\label{qn-aihs-reflexive} 
Every triangularizable quasinilpotent operator on a reflexive Banach space has an almost-invariant half-space.
\end{thm}
\begin{proof}
Let $T\in\cB(\cX)$ be a triangularizable quasinilpotent operator. Clearly, we may assume that the nullity of $T$ is  finite, for otherwise, any half-space of $\ker\, T$ is an invariant half-space for $T$. First, we will show that, without loss of generality, we may assume that $T$ has dense range.

Denote $\cY_n=\overline{T^n\cX}$, $\cY_0= \cX$. Since each $\cY_n$ is $T$-invariant, we may assume that each $\cY_n$ is of finite codimension in~$\cX$. Write $\cX=\cY_1\oplus\cZ$ for some finite-dimensional space~$\cZ$. We have:
$$
\cY_2=\overline{T^2 \cX}=\overline{T\cY_1},\mbox{ so}
$$
$$
\cY_1=\overline{T \cX}=\overline{T\cY_1+T\cZ}=\overline{T\cY_1}+T\cZ=\cY_2+T\cZ,
$$
since $T\cZ$ is finite-dimensional. Similarly, we get
$$
\cY_{n}=\cY_{n+1}+T^{n}\cZ
$$
for all $n\in\bbN$.

Now write the space $T^n\cZ$ as $W_{n1}\oplus W_{n2}$ where $W_{n1}\subseteq\cY_{n+1}$ and $W_{n2}\cap\cY_{n+1}=\{0\}$. Then $T^{n+1}\cZ=T(W_{n1}\oplus W_{n2})=TW_{n1}+TW_{n2}$, with $TW_{n1}\subseteq\cY_{n+2}$. It follows that $\mathrm{codim}_{\cY_n}Y_{n+1}$ is a decreasing sequence. Hence, it must stabilize at some~$n$. Restricting $T$ to $\cY_n$, we may assume without loss of generality that this happens for $n=1$. Thus
\[
\cY_{n}=\cY_{n+1}\oplus T^{n}\cZ,\quad\dim T^n\cZ=\dim\,  \cZ
\]
for all $n\in\bbN$.

If $\cZ=0$, then $T$ has dense range. Otherwise, pick a non-zero $z\in\cZ$. Then $T^nz\in T^n\cZ\subseteq\cY_n\setminus\cY_{n+1}$. It follows that the sequence $(T^nz)_{n\ge 1}$ has the property that $T^nz\not\in[T^iz]_{i>n}$. Since $T$ is quasinilpotent, we get by Lemma~\ref{minimal-sequences}, Theorem~\ref{donoghue-aihs} and Remark~\ref{redundant-condition} that $T$ has an almost-invariant half-space.

So, we may assume that $T$ has dense range. Then, clearly, $T^{*}$ is injective and quasinilpotent. Also, if $(\cM_\alpha)$ is a triangularizing chain for $T$ then $(\cM^\perp_\alpha)$ is a triangularizing chain for $T^*$. Thus, $T^*$ is triangularizable. By Theorem~\ref{qn-aihs}, $T^*$ has an almost-invariant half-space. Hence, by \cite[Proposition~1.7]{APTT}, so does $T=T^{**}$.
\end{proof}


\begin{rem}
The affirmative answer to the Invariant Subspace Problem for quasinilpotent operators on Hilbert spaces would imply that every quasiniplotent operator in $\cB(\hilb)$ is triangularizable. Hence, Theorem~\ref{qn-aihs-reflexive} shows that, when we restrict our attention to the quasinilpotent operators on Hilbert spaces, the problem of the existence of almost-invariant half-spaces is a \emph{weakening} of the Invariant Subspace Problem. The same remark can be made about the reflexive spaces.
\end{rem}


Our next goal is finding almost-invariant half-spaces for polynomially compact operators. Recall that an operator $T$ is called \term{polynomially compact} if there is a non-zero polynomial $p$ such that $p(T)$ is compact. 

%
%
%
%
%


\begin{lem}\label{disconnected-spectrum} 
Let $T$ be an operator on an infinite-dimensional Banach space $\cX$. If $\sigma(T)$ has infinitely many connected components then $T$ has an invariant half-space.
\end{lem}
\begin{proof}
The set $\sigma(T)$ must contain infinitely many connected components that are relatively open sets in $\sigma(T)$. Denote all such components by $\{ \sigma_n\}_{n=1}^\infty$. For each $n\in\mathbb N$, there is a Riesz projection for $\sigma_n$, that is, an idempotent  $E_n$ such that $E_nT=TE_n$, $\sigma(TE_n|_{E_n \cX})=\sigma_n$, and $\sigma(T(I-E_n)|_{(I-E_n)\cX})=\sigma(T)\setminus\sigma_n$. Observe that, since $\sigma(T)\setminus\sigma_n$ is infinite for each~$n$, $(I-E_n)\cX$ is infinite-dimensional for each~$n$. As such, if $E_n \cX$ is infinite-dimensional for some $n$, then $E_n \cX$ is a $T$-invariant half-space.

Thus we may assume that $E_n \cX$ is finite-dimensional for each $n\in\mathbb N$. It follows that each $\sigma_n$ is a singleton, $\sigma_n=\{\lambda_n\}$, and $\lambda_n$ is an eigenvalue.

For each $\lambda_n$, pick a non-zero vector $x_n$ such that $Tx_n=\lambda_nx_n$. Define $\cY =[x_{2n}]_{n=1}^\infty$. We claim that $\cY$ is a $T$-invariant half-space.

The $T$-invariance is obvious. Also, $\cY$ is  clearly  infinite-dimensional. We need to prove that $\codim\,\cY=\infty$.

Let $m\in\mathbb N$ be arbitrary. Let $F$ be a Riesz projection corresponding to $\sigma_1\cup\sigma_3\cup\dots\cup\sigma_{2m-1}$. We claim that $\cY \subseteq (I-F) \cX$.

Clearly, it is enough to prove that each $x_{2n}$ belongs to $(I-F) \cX$. Denote $\cX_1=F \cX$, $\cX_2=(I-F) \cX$, $T_1=T_{\cX_1}$, $T_2=T_{\cX_2}$, so that $\cX= \cX_1\oplus \cX_2$ and $T=T_1\oplus T_2$. Write $x_{2n}$ as $(x_{2n}^{(1)},x_{2n}^{(2)})$, relative to this decomposition of $\cX$. Then $Tx_{2n}=\big(T_1x_{2n}^{(1)},T_2x_{2n}^{(2)}\big)=\lambda_{2n}\big(x_{2n}^{(1)},x_{2n}^{(2)}\big)$. Since $\lambda_{2n}$ is not in the spectrum of $T_1$, $x_{2n}^{(1)}=0$. It follows that $x_{2n}\in \cX_2$.

Since $m\in\mathbb N$ was arbitrary, this shows that $\cY$ is a half-space.
\end{proof}



\begin{rem} 
The half-space constructed in Lemma~\ref{disconnected-spectrum} is spanned by the eigenvectors corresponding to even-numbered isolated eigenvalues of~$T$. It should be noted that if we drop the condition that the eigenvalues are isolated, the construction may not work. For example, let $T$ be the backward shift on $\ell_2$. Every element of the open unit disk is an eigenvalue for $T$. For each $\lambda$ in the open unit disk, let $x_\lambda=(\lambda,\lambda^2,\lambda^3,\dots)$, so that $Tx_\lambda=\lambda x_\lambda$. Then for any subset $\Lambda$ of the open unit disk having an accumulation point also in the open unit disk, the span of $\{x_\lambda\mid\lambda\in\Lambda\}$ is dense in~$\ell_2$. Indeed, let $\cY=\bigvee\{x_\lambda\mid\lambda\in\Lambda\}$. Let $y=(y_1,y_2,y_3,\dots)\in\cY^{\perp}$ be arbitrary. Consider the  function $f(z)=\sum_{i=1}^\infty \overline{y_i}z^i$, which is analytic in the unit disk. For each $\lambda\in\Lambda$, we have $f(\lambda)=\sum_{i=1}^\infty \overline{y_i}\lambda^i=\langle x_\lambda,y\rangle=0$. Thus, the set of zeros of $f$ is infinite, with an accumulation point in the interior of the unit disk. Therefore, $f$ and, hence, $y$ must be equal to zero.
\end{rem}


\begin{rem}\label{connected-spectrum} 
It follows from Lemma~\ref{disconnected-spectrum} that, when looking to prove the existence of almost-invariant half-spaces for an operator $T$, one may assume that the spectrum of $T$ is connected. Indeed, assume $\sigma(T)$ is not connected. By Lemma~\ref{disconnected-spectrum}, we may assume that $\sigma(T)=\sigma_1\cup\sigma_2\cup\dots\cup\sigma_n$ where each $\sigma_i$ is connected. For each $i=1,\dots,n$, consider a Riesz projection $E_i$ corresponding to $\sigma_i$. That is, $E_iT=TE_i$, $\sigma(TP_i|_{E_i \cX})=\sigma_i$, and $E_1+E_2+\dots+E_n=I$. Denote by $\cX_i$ the subspace $E_i \cX$. It is clear that we may assume without loss of generality that only one of $\cX_i$'s, say, $\cX_1$, is infinite-dimensional. Consider the operator $S=E_1TE_1\in \cB(\cX_1)$. If $S$ has an almost-invariant half-space then so does $T$. Conversely, if $\cY \subseteq \cX$ is a $T$-almost-invariant half-space then $\cY \cap \cX_1$ is an $S$-almost-invariant half-space.
\end{rem}


\begin{rem}\label{restr-triang}
Note that the operator $S$ in Remark~\ref{connected-spectrum} is triangularizable if $T$ is triangularizable. 
\end{rem}




\begin{thm}\label{triang}
Every triangularizable operator $T$ with countable spectrum acting on a reflexive Banach space $\cX$ has an almost-invariant half-space.
\end{thm}

\begin{proof}
By Remarks~\ref{connected-spectrum} and~\ref{restr-triang}, we may assume that $\sigma(T)$ is a singleton. Subtracting a scalar multiple of identity, we may assume that $T$ is quasinilpotent. The statement of the theorem now follows from Theorem~\ref{qn-aihs-reflexive}.
\end{proof}


\begin{cor}\label{p-compact}
Every polynomially compact operator on a reflexive Banach space $\cX$ admits an almost-invariant half-space.
\end{cor}

\begin{proof}
Let $T$ be the polynomially compact operator. 
If $p$ is a polynomial such that $p(T)\in\cK(\cX)$ and $\pi:\cB(\cX)\to\cB(\cX)/\cK(\cX)$ is the canonical quotient map, then $p(\pi(T))=0$, and so the essential spectrum $\sigma(\pi(T))$ of $T$ is finite. From this and the Putnam-Schechter theorem (see~\cite{Putn62} or \cite[Theorem~6.8]{ACFA}), it follows that the spectrum of $T$ consists of countably many points, and the finitely many elements of $\sigma(\pi(T))$ are the only possible accumulation points of $\sigma(T)$. Since every polynomially compact operator is triangularizable (see \cite{BR1966, Hal1966}), it follows from Theorem~\ref{triang} that $T$ has an almost-invariant half-space.
\end{proof}


\begin{rem}
In view of Theorem~\ref{p-compact} and the well-known theorem of Lomonosov (see~\cite{Lom73}; see also, e.g., \cite[Theorem~10.19]{AA02}) that all operators commuting with a non-zero compact operator have hyperinvariant subspaces, one may ask if compact operators always admit \emph{almost hyperinvariant half-spaces}, that is, half-spaces which are almost-invariant under all operators commuting with a given compact operator. This question has a negative answer. Indeed, let $D$ be a Donoghue operator (see the Introduction for the definition and properties of Donoghue operators). Observe that $D$ is compact. If $D$ had an almost hyperinvariant half-space then this half-space would be almost-invariant for the norm-closed algebra $\mathcal A$ generated by $D$. However, it follows from \cite[Theorem~3.6]{Popov2010} that $\mathcal A$ has no almost-invariant half-spaces.
\end{rem}


For the remainder of this section, we shall restrict our attention to separable Hilbert spaces.

\bigskip

Given the results of this section, it makes sense to pose the following question: \emph{does every triangularizable operator admit an almost-invariant half-space}? 

We will consider a special case of triangularizable operators: bitriangular operators. Recall that an operator $T$ on a separable Hilbert space $\cH$ \term{triangular} if it has an upper triangular matrix with respect to some orthonormal basis indexed by the natural numbers. That is, there exists an orthonormal basis $(e_n)_{n=1}^\infty$ for $\cH$ such that $\langle Te_j,e_i\rangle=0$ whenever $i>j$. We refer the reader to~\cite[Chapter~3]{Herr89}, \cite{DH90} and~\cite{Herr91} for more information about the triangular operators. The operator $T$ is called \term{bitriangular} if both $T$ and $T^*$ are triangular, perhaps with respect to different orthonormal bases.

We will use some definitions and notations from~\cite{DH90}. For an operator $A\in\mathcal B(\cH)$ and an integer $k \ge 1$, denote by $\ker(A,k)$ the space $\ker A^k\ominus\ker A^{k-1}$. By $\nul(A,k)$ we will denote the dimension of $\ker(A,k)$, and finally, the symbol $\alpha(A,k)$ will stand for the difference $\nul(A,k)-\nul(A,k+1)$. Here the difference $\infty-\infty$ is considered to be~$\infty$.

Observe that if $T$ is an operator acting on a finite-dimensional space, the number $\nul(T-\lambda,k)$ counts the number of Jordan blocks corresponding to $\lambda$ of size at least~$k$, while the number $\alpha(T-\lambda,k)$ counts the number of blocks for $\lambda$ of size exactly~$k$. This motivates the following definition from~\cite{DH90}: if $T\in\mathcal B(\cH)$, the \term{canonical Jordan model} for $T$ is defined as
$$
J(T)=\bigoplus_{\lambda\in\sigma_p(T)}J(T,\lambda)=
\bigoplus_{\lambda\in\sigma_p(T)}\bigoplus_{k\ge 1}(\lambda I_k+J_k)^{(\alpha(T-\lambda,k))}.
$$
Here $I_k$ is the identity $k\times k$-matrix and $J_k$ is the nilpotent $k\times k$-matrix 
$$
J_k=\left[\begin{array}{ccccc}
0 & 1 & 0 & \dots & 0\\
0 & 0 & 1 & \dots & 0\\
  & & & \dots & \\
0 & 0 & 0 & \dots & 1\\
0 & 0 & 0 & \dots & 0
\end{array}\right],
$$
and the symbol $A^{(j)}$ means the direct sum of $j$ copies of the operator~$A$. The direct sum is taken in the Hilbert space sense: 
the underlying space is a Hilbert space direct sum of the corresponding summands (in particular, each summand is orthogonal to any other summand).

Recall that an operator $A\in\mathcal B(\cH)$ is called a \term{quasiaffinity} if $A$ is injective and has dense range. Two operators $T$ and $S$ in $\mathcal B(\cH)$ are called \term{quasisimilar} if there exist two quasiaffinities $A$ and $B$ such that $AT=SA$ and $TB=BS$. It is known (see~\cite{H72}) that quasisimilarity preserves the existence of hyperinvariant subspaces. However, it may not preserve the structure of the lattice of hyperinvariant subspaces (see~\cite{H78}). It is not known whether or not quasisimilarity preserves the existence of invariant subspaces.

The following statement about bitriangular operators is the main result of~\cite{DH90}.

\begin{thm}\label{BT-jordan-form}
{\rm\cite[Theorem~4.6]{DH90}} Every bitriangular operator is quasisimilar to its canonical Jordan model.
\end{thm}


We are now ready to state our result about bitrangular operators.

\begin{thm}\label{bitriangular}
If $T$ is a bitriangular operator then either $T$ can be written as $\lambda I+F$ where $F$ is a finite rank operator or $T$ has a hyperinvariant half-space. In particular, $T$ has an invariant half-space.
\end{thm}
\begin{proof}
Let $T$ be a bitriangular operator that cannot be written in the form $\lambda I+F$ where $F$ is a finite rank operator. Following \cite[Lemma~4.5]{DH90}, for a subset $\Gamma\subseteq\mathbb C$, set
\[
\cH(T,\Gamma)=\bigvee\{\ker(T-\lambda)^k\mid\lambda\in\Gamma,\ k\in\mathbb N\}.\]

Since $T$ is bitriangular, the point spectrum $\sigma_p(T)$ is non-empty and countable (see \cite[Theorem~3.1]{DH90}). Write $\sigma_p(T)$ as a (perhaps finite) sequence of distinct complex numbers $\lambda_1,\lambda_2,\lambda_3\dots$ . For each $n$, define 
$$
T_n=\bigoplus_{k\ge 1}(\lambda_n I_k+J_k)^{(\alpha(T-\lambda_n,k))}.
$$
The canonical Jordan model of $T$ is $J(T)=\bigoplus_nT_n$. By Theorem~\ref{BT-jordan-form}, $T$ is quasisimilar to $J(T)$. Fix two quasiaffinities $A$ and $B$ such that $AT=J(T)A$ and $TB=BJ(T)$. 

Notice that the point spectrum of each summand in $J(T)$ is a singleton: $\sigma_p(T_n)=\{\lambda_n\}$.

Assume first that $\sigma_p(T)$ is infinite. Consider $\Gamma=\{\lambda_{2n}\}_{n=1}^\infty$. By \cite[Lemma~4.5]{DH90}, the spaces $\cH(T,\Gamma)$ and $\cH(T,\mathbb C\setminus\Gamma)$ are $T$-hyperinvariant, $\cH=\cH(T,\Gamma)\vee\cH(T,\mathbb C\setminus\Gamma)$ and $\cH(T,\Gamma)\cap\cH(T,\mathbb C\setminus\Gamma)=\{0\}$. Also, $\overline{A\cH(T,\Gamma)}=\cH(J(T),\Gamma)$. It is clear that $\cH(J(T),\Gamma)$ is infinite-dimensional. Therefore, so is $\cH(T,\Gamma)$. 

Analogously, the space $\cH(T,\mathbb C\setminus\Gamma)$ is infinite-dimensional, too. It follows that $\cH(T,\Gamma)$ is a $T$-hyperinvariant half-space.

Assume now that $\sigma_p(T)$ is finite: $\sigma_p(T)=\{\lambda_1,\dots,\lambda_n\}$. Abbreviate $\cH(T,\{\lambda_i\})$ as $\cH(T,\lambda_i)$. Then $\cH=\cH(T,\lambda_1)\vee\dots\vee\cH(T,\lambda_n)$, where each $\cH(T,\lambda_i)$ is $T$-hyperinvariant, and each pair of the summands in this decomposition has trivial intersection. If there are two indices $i\ne j$ such that $\dim\big(\cH(T,\lambda_i)\big)=\dim\big(\cH(T,\lambda_i)\big)=\infty$, then either of them is a $T$-hyperinvariant half-space. 

Observe that if all of $\cH(T,\lambda_i)$ is finite-dimensional then $J(T)$ is finite rank. Hence, by injectivity of $A$, we can conclude that $T$ is finite rank, too.  Therefore, we may assume without loss of generality that exactly one of the $\cH(T,\lambda_i)$, $i \ge 1$ -- say $\cH(T,\lambda_1)$ -- is infinite-dimensional. It follows from \cite[Lemma~2.3]{DH90} and \cite[Proposition~2.9]{Herr91} that $T|_{\cH(T,\lambda_1)}$ is a bitriangular operator. Therefore, we may assume that the point spectrum of $T$ is a singleton.

Denote this unique element of $\sigma_p(T)$ by $\lambda$. Since scalar perturbations do not change hyperinvariant subspaces, we may assume that $\lambda=0$. Also, since $J(T)$ is not finite rank, it must have infinitely many blocks of size $\ge 2$. In particular, $J(T)$ acts on an infinite-dimensional space and has infinite-dimensional kernel. By injectivity of $B$, $\ker T$ is infinite-dimensional, too. By \cite[Proposition~3.2]{DH90}, $\cH=\bigvee_{k\ge 1}\ker T^k$. It follows from the definition of $J(T)$ that if $\ker T$ were finite codimensional, $J(T)$ would have only finitely many blocks of size $\ge 2$, a contradiction.   Since we assumed that $T$ is not finite-rank at the outset, $\ker\, T$ is a half-space, which is easily seen to be hyperinvariant for $T$.
\end{proof}


\section{Operator algebras with common almost-invariant half-spaces} \label{sec3}


\subsection{} \label{sec3.1}
Our goal in this section is to show that if $\cA$ is a norm-closed algebra of operators on a Banach space $\cX$, and if $\cY$ is an almost-invariant subspace for $\cA$, then $\cA$ admits an invariant subspace.  We begin with a Lemma which will also be useful in the next section.


\begin{lem} \label{lem4.43aBanach}
Let $\cX$ and $\cY$ be Banach spaces.   Suppose that  $T = \begin{bmatrix} A & X \\ C & B \end{bmatrix} \in \cB(\cX \oplus \cY)$.  If $\kappa \in \bbN$ and $\mathrm{rank}\, T = \mathrm{rank}\, C = \kappa < \infty$, then $X$ is uniquely determined by $A, C$ and $B$. If, in addition, $C$ is invertible (which can only happen when $\dim(\cX\oplus\cY)<\infty$), then $X=AC^{-1}B$.
\end{lem}

\begin{pf}
Since $C: \cX\to \cY$ has rank $\kappa < \infty$, $\cW_1 := \ran\, C$ is topologically complemented in $\cY$.  Thus  we may decompose $\cY = \cW_1 \oplus \cW_2$ for some closed subspace $\cW_2$ of $\cY$.   Similarly, $\cV_2:= \ker\, C$ is finite-codimensional in $\cX$, and so we may decompose $\cX = \cV_1 \oplus \cV_2$ for some closed subspace $\cV_1$ of $\cX$.  With respect to these decompositions of $\cX$ and $\cY$, we may then write 
\[
C = \bordermatrix{\ & \cV_1 &  \cV_2 \cr \cW_1 & C_0 & 0 \cr \cW_2 & 0 & 0 }, \] 
where $C_0$ is invertible.   From rank considerations, it then follows that $T$ must be of the form:
\[
T = \left[ \begin{array}{c|c} \begin{array}{cc} * & 0 \\ * & 0 \end{array} & \begin{array}{cc} * & * \\ * & * \end{array} \\    --- & --- \\ \begin{array}{cc} C_0 & 0 \\ 0 & 0 \end{array} & \begin{array}{cc} * & * \\ 0 & 0 \end{array} \end{array} \right] \]
relative to the decomposition $\cX \oplus \cY = \cV_1 \oplus \cV_2 \oplus \cW_1 \oplus \cW_2$.

Let $E_0: \cW_1 \to \cV_1$ denote the inverse of $C_0$, so that $E_0 C_0 = I|_{\cV_1}$ and $C_0 E_0 = I|_{\cW_1}$.   We can extend $E_0$ to a continuous linear map 
\[
E = \begin{bmatrix} E_0 & 0 \\ 0 & 0 \end{bmatrix}, \]
with the block-matrix decomposition corresponding to that of an operator from $\cY = \cW_1 \oplus \cW_2$ to $\cX = \cV_1 \oplus \cV_2$.
Then $E C \in \cB(\cX)$ is simply the  projection of $\cX$ onto $\cV_1$, and the operator matrix for $T$ above shows that $A E C = A$. From this we conclude that

\[
\begin{bmatrix} I_\cX & -A E \\ 0 & I_\cY \end{bmatrix} \begin{bmatrix} A & X \\ C & B \end{bmatrix} = \begin{bmatrix} 0 & X - A E B \\ C & B \end{bmatrix}. \] 

Since the operator $\begin{bmatrix} I_\cX & -A E \\ 0 & I_\cY \end{bmatrix}$ is clearly invertible, it preserves rank.   But 
\[
\mathrm{rank}\, C = \kappa = \mathrm{rank}\, \begin{bmatrix} 0 & X - A E B \\ C & B \end{bmatrix} \]
then implies that $X - A E B = 0$.  Since $E$ is clearly uniquely determined by $C$, the statement of the Lemma is proven.
\end{pf}

\bigskip

The following theorem is the main result of this section.

\begin{thm} \label{thm3.3Banach}
Let $\cX$ be a Banach space and $\cA \subseteq \cB(\cX)$ be a norm-closed algebra of operators.   Suppose that $\cY$ is a closed, topologically complemented half-space of $\cX$ and that for each $T \in \cA$, there exists a finite-dimensional subspace $\cM_T$ of $\cX$ so that $T \cY \subseteq (\cY + \cM_T)$.   Then $\cA$ admits an invariant half-space $\cL$.   Moreover, there exists a finite-dimensional subspace $\cF$ of $\cX$ so that $\cL$ is a finite-codimensional subspace of $\cY + \cF$.
\end{thm}

\smallskip

\noindent{\textbf{Remark.}}  Informally, one could think of $\cL$ as lying ``within a finite-dimensional subspace" of $\cY$.
\smallskip

\begin{pf}
Clearly we may assume without loss of generality that $\cA$ is unital. Let $E: \cX \to \cX$ be a continuous idempotent map with range equal to $\cY$. The kernel $\cZ$  of $E$ (equal to the range of $I-E$) serves as a topological complement to $\cY$. 

Consider the set $\cS := \{ (I-E) T E: T \in \cA\}$ as a subspace of $\cB(\cY, \cZ)$. Observe that by Theorem~2.7 of~\cite{Popov2010}, there exists $1 \le \kappa < \infty$ such that 
\[
\max_{T \in \cA} \mathrm{rank}\, \big((I-E) T E\big)= \kappa. \]
Fix an element $T \in \cA$ such that $\mathrm{rank} ((I-E) T E) = \kappa$.   

Write $T = \begin{bmatrix} T_1 & T_2 \\ T_3 & T_4 \end{bmatrix}$ relative to the decomposition $\cX = \cY \oplus \cZ$, so that $T_3~=$\linebreak $(I-E)~T~E|_{E \cX}$.   Now $\cW_1 := \ran\, T_3$ is a finite-dimensional subspace of $\cZ$, and hence is complemented in that space, say $\cZ = \cW_1 \oplus \cW_2$.  Similarly, $\cV_1 := \ker\, T_3$ is a finite-codimensional subspace of $\cY$, and thus is complemented there, say $\cY = \cV_1 \oplus \cV_2$.   With respect to these decompositions, we may write
\[
T_3 = 
\bordermatrix{\ & \cV_1 &  \cV_2 \cr \cW_1 & 0 & T_{3 2} \cr \cW_2 & 0 & 0 }. \] 
Furthermore, $T_{3 2}$ is an injective map from $\cV_2$ onto $\cW_1$, and thus is an invertible element of $\cB(\cV_2, \cW_1)$.
We claim that $S \in \cS$ implies that $S \cV_1 \subseteq \cW_1$.   

Let $L \in \cS$, and write 
\[
L = \bordermatrix{\ & \cV_1 &  \cV_2 \cr \cW_1 & L_1 & L_2 \cr \cW_2 & L_3 & L_4 }. \]
Then, for all $\alpha \in \bbC$, $L + \alpha T_{3} \in \cS$, and thus $\mathrm{rank} (L + \alpha T_{3}) \le \kappa$.  In fact, this has rank equal to $\mathrm{rank}\, T_{3 2} = \kappa$ for all but finitely many values of $\alpha$, since $\mathrm{rank}\, (L_2 + \alpha T_{3 2}) = \kappa$ for all but finitely many values of $\alpha \in \bbC$.  By Lemma~\ref{lem4.43aBanach}, for all but finitely many values of $\alpha$, 
\[
L_3 = L_4 (L_2 + \alpha T_{3 2})^{-1} L_1.\]
By letting the absolute value of $\alpha$ tend to infinity, we see that $\norm (L_2 + \alpha T_{3 2})^{-1} \norm$ tends to zero, and thus $L_3$ tends to zero.   But $\norm L_3 \norm$ is constant, a contradiction unless $L_3 = 0$.   This proves the claim.

This allows us to write each element of the algebra $\cA$ in the form
\[
\bordermatrix{\ & \cV_1 & \cV_2 &  \cW_1 & \cW_2 \cr \cV_1 & * & * & * & * \cr \cV_2 & * & * & * & * \cr \cW_1 & * & * & * & * \cr \cW_2 & 0 & * & * & * }. \] 

Now consider the space $\cL :=\overline{\cA\cV_1}\subseteq\cX$. Clearly, $\cL$ is $\cA$-invariant. Also, since $\cA$ is unital, 	$\cV_1 \subseteq\cL$, so that $\cL$ is infinite-dimensional. On the other hand, it follows from the above matrix form of operators in $\cA$ and the fact that the space $\cY\oplus\cW_1$ is closed that $\cL=\overline{\cA(\cV_1)}\subseteq\cV_1\oplus\cV_2\oplus\cW_1=\cY\oplus\cW_1$. Since $\cX=\cY\oplus\cW_1\oplus\cW_2$ and $\dim\cW_2=\infty$, the space $\cL$ is of infinite codimension. Therefore, $\cL$ is a half-space.  The final statement of the theorem is easy to verify. 
\end{pf}

The next Corollary follows from Theorem~\ref{thm3.3Banach} and the fact that  every closed subspace of a Hilbert space is topologically complemented.  Observe that the Corollary applies, in particular, to any $C^*$-subalgebra of $\bofh$, whose invariant half-spaces are automatically reducing.

\begin{cor} \label{cor3.4}
Let $\cA \subseteq \cB(\cH)$ be a norm-closed algebra of operators on $\cH$.  Let $\cY$ be a closed half-space of $\cH$ and suppose that for each $T \in \cA$, there exists a finite-dimensional subspace $\cM_T$ of $\hilb$ so that $T \cY \subseteq \cY + \cM_T$.   Then $\cA$ admits an invariant half-space.
\end{cor}


\section{almost-invariant subspaces in masas}\label{sec4}

\subsection{} \label{sec4.1}
Let us first establish some notation that will be used throughout this section.
By $\mathcal H$ we shall denote a separable Hilbert space,  and $\cD \subseteq \bofh$ will denote a maximal abelian, selfadjoint subalgebra (i.e. a masa) of $\bofh$.      

If $(X,\Sigma,\mu)$ is a measure space and $\hilb = L^2(X, \mu)$, then the separability of $\hilb$ implies that $\mu$ is $\sigma$-finite.    Moreover, the set $\cD = \{ M_f: f \in L^\infty(X, \mu)\}$, where $M_f: L^2(X, \mu) \to L^2(X, \mu)$ is the operator $M_f g = f g$, is a masa in $\cB(L^2(X, \mu))$.   As is well-known, given any masa $\cD$ in $\bofh$, there exists a measure space $(X, \mu)$ so that $\hilb$ is unitarily equivalent to $L^2(X, \mu)$ and $\cD$ is unitarily equivalent to the masa $\{ M_f : f \in L^\infty(X, \mu)\}$ defined above.

For $E\subseteq X$  a measurable set, $E^c$ will denote the complement of $E$ in $X$, and the projection 
\[
\begin{array}{ccc}
L^2(X, \mu)& \to &  L^2(X, \mu)\\ 
 f &\mapsto &\chi_E\cdot f,
\end{array}\]
 will be denoted by $P_E$. Also, if $f,g\in L^2(X, \mu)$  then the rank-one operator $h\mapsto\langle h,g\rangle f$ in $\cB(L^2(X, \mu))$ will be denoted by $f\otimes g^*$.  

Let $\cD \subseteq \bofh$ be  a masa.  We shall denote by $\cP(\cD)$ the set of orthogonal projections in $\cD$. Suppose that $D \in \cD$ and that $F \in \cF(\cH)$ has rank $n < \infty$.   Let $T = D + F$.   It is clear that if $P \in \cP(\cD)$, then 
\[ 
\mathrm{rank}\, (T P - P T P) = \mathrm{rank} (I - P) F P \le n < \infty. \]
Our goal in this section is to provide a converse to this result.   
More specifically, wiith $\cD \subseteq \bofh$ a masa as above, suppose that  $T \in \bofh$ is such that every projection $P \in \cP(\cD)$ is almost-invariant for $T$.  We will show that this implies that  $T = D + F$ for some  $D \in \cD$ and $F \in \cF(\hilb)$.

\bigskip
We begin with a technical lemma.


\begin{lem}\label{sets-manipulation} 
Let $(X,\Sigma,\mu)$ be a measure space. Suppose that $E\in\Sigma$, $T\in\mathcal B(L^2(X, \mu))$, and $f_1,\dots,f_n\in L^2(X, \mu)$ are such that the set $\{P_{E^c}TP_Ef_k\}_{k=1}^n$ is linearly independent. 
Then there exists $\varepsilon>0$ such that for all $A\subseteq E^c$ and $B\subseteq E$ in $\Sigma$ with $\mu(A),\mu(B)<\varepsilon$, the set $\{P_{E^c\setminus A}TP_{E\setminus B}f_k\}_{k=1}^n$ is linearly independent.
\end{lem}

\begin{proof}
Denote the space $\mathrm{span}\{f_1,\dots,f_n\}$ by~$\cM$. Assume that there are sequences $(A_k)$, $(B_k)$ of measurable sets such that $\mu(A_k),\mu(B_k)\to 0$ and $\mathrm{rank}(P_{E^c\setminus A_k}TP_{E\setminus B_k}|_{\cM})\le ~n~-~1$.   Observe that the sequence $(P_{E^c\setminus A_k}TP_{E\setminus B_k})$ converges to $P_{E^c}TP_{E}$ in the strong operator topology. In particular, $P_{E^c\setminus A_k}TP_{E\setminus B_k}|_\cM\overset{SOT}{\longrightarrow}P_{E^c}TP_{E}|_\cM$. However, $\cM$ is finite-dimensional, hence $P_{E^c\setminus A_k}TP_{E\setminus B_k}|_\cM \overset{\norm{\cdot}\norm}{\longrightarrow}P_{E^c}TP_{E}|_\cM$. It follows from the lower semicontinuity of the rank that $\mathrm{rank}(P_{E^c}TP_{E}|_\cM)\le n-1$, contrary to the assumptions of the lemma.
\end{proof}


\subsection{} \label{sec4.3}

As pointed out above, a necessary condition for an operator $T$ to be expressible in the form $T = D + F$ with $D \in \cD$ and $F$ a finite-rank operator is the existence of an integer $n$ (which may be chosen to be $\mathrm{rank}\, F$) so that  
\[
\sup_{P \in \cP(\cD)} \mathrm{rank}\, (I-P) T P \le  n < \infty.  \]
Our first goal therefore is to show that such an integer always exists.

\bigskip

\begin{thm}\label{masa-bounded-defect}
Let $\mathcal D$ be a masa in $\mathcal B(\mathcal H)$ and $T\in\mathcal B(\mathcal H)$. Suppose that for every projection $P\in\mathcal D$ onto a half-space $\cY_P$, the space $\cY_P$ is almost-invariant under $T$. Then there exists $\kappa\in\mathbb N$ such that $\mathrm{rank}(TP-PTP)\le\kappa$ for all projections $P\in\mathcal D$ (including the projections onto finite-dimensional or finite-codimensional subspaces).
\end{thm}

\begin{proof}
Let us choose a measure space $(X, \Sigma, \mu)$ such that $\mathcal B(\mathcal H) = L^2(X, \mu)$ and $\mathcal D = \{ M_f: f \in L^\infty(X, \mu)\}$ is the set of multiplication operators on $ L^2(X, \mu)$. We may next decompose $L^2(X, \mu) = L^2(X_d, \mu_d)\oplus L^2(X_c, \mu_c)$ where the measure $\mu_d$ is purely atomic and $\mu_c$ is atomless. Relative to this decomposition of $\mathcal B(\mathcal H)$, $T$ can be written in the matrix form
$$
\left[\begin{matrix}
T_1 & T_2\\
F & T_4
\end{matrix}\right],
$$
where $\mathrm{rank}(F)<\infty$ (and $\mathrm{rank}\, {T_2} < \infty$, though we shall not need this). The above decomposition of $L^2(X, \mu)$ induces a decomposition of the algebra $\cD$ as  $\cD = \cD_d + \cD_c$, where  $P_d$ (resp. $P_c$) is the orthogonal projection of $L^2(X, \mu)$ onto $L^2(X_d, \mu_d)$ (resp. $L^2(X_c, \mu_c)$), $\mathcal D_d = P_d \cD$ and $\cD_c = P_c \cD$.    If we can show that there are $\kappa_1,\kappa_2\in\mathbb N$ such that $\mathrm{rank}(T_1P-PT_1P)\le\kappa_1$ and $\mathrm{rank}(T_4Q-QT_4Q)\le\kappa_2$ for all projections $P\in L^2(X, \mu_d)$ and $Q \in L^2(X, \mu_c)$, then $\kappa=\kappa_1+\kappa_2+\mathrm{rank}(F)$ will satisfy the conclusion of the theorem. So, we may assume without loss of generality that $\mu$ is either purely atomic or atomless. We consider three cases.

\emph{Case 1}: $\mathcal H= L^2(X, \mu)$ where $\mu$ is an atomless, finite measure. 

Suppose that $\{\mathrm{rank}(TP_E-P_ETP_E)\mid E\mbox{  measurable}\}$ is unbounded. In particular, there exists a measurable set $E_1$ such that $P_{E_1^c}TP_{E_1}\ne 0$. That is, there exists $f_{11}\in\mathcal H$ such that $\mathrm{supp}\,f_{11}\subseteq E_1$ and $P_{E^c_1}Tf_{11}\ne 0$.

Observe that, since $\mu$ is finite and atomless, the condition of unboundedness of the set $\{\mathrm{rank}(TP_E-P_ETP_E)\mid E\mbox{  measurable}\}$ implies that, given $N\in\mathbb N$ and $\varepsilon>0$ there exist $F,G\in\Sigma$ such that $F\cap G=\varnothing$, $\mu(F),\mu(G)<\varepsilon$, and $\mathrm{rank}(P_FTP_G)\ge N$.

Put $Q_1=E_1$. By Lemma~\ref{sets-manipulation}, there exists $ \varepsilon_1 > 0$ such that $P_{Q_1^c\setminus A}TP_{Q_1\setminus B}f_{11}\ne 0$ for all measurable sets $A$ and $B$ satisfying $\mu(A)$, $\mu(B)<\varepsilon_1$. By the observation above, we can find disjoint sets $E_2$ and $F_2$ such that $\mu(E_2)$, $\mu(F_2)<\frac{\varepsilon_1}{2}$ and $\mathrm{rank}(P_{F_2}TP_{E_2})\ge 2$. That is, there are $f_{21}$ and $f_{22}$ in $\mathcal H$ such that $\mathrm{supp}\,f_{21}$, $\mathrm{supp}\,f_{22}\subseteq E_{2}$ and $\{P_{F_2}Tf_{21},P_{F_2}Tf_{22}\}$ is linearly independent.

Put $Q_2=(Q_1\cup E_2)\setminus F_2$. Then $E_2\subseteq Q_2$ and $F_2\subseteq Q_2^c$, hence the set $\{P_{Q_2^c}TP_{Q_2}f_{2i}\}_{i=1}^2$ is linearly independent. By Lemma~\ref{sets-manipulation}, there exists $0<\varepsilon_2<\frac{\varepsilon_1}{2}$ such that the set $\{P_{Q_2^c\setminus A}TP_{Q_2\setminus B}f_{2i}\}_{i=1}^2$ is linearly independent for all measurable sets $A$ and $B$ satisfying $\mu(A),\mu(B)<\varepsilon_2$. Again, fix disjoint sets $E_3$ and $F_3$ such that $\mu(E_3)$, $\mu(F_3)<\frac{\varepsilon_2}{2}$ and $\mathrm{rank}(P_{F_3}TP_{E_3})\ge 3$. There exist $f_{31}$, $f_{32}$, and $f_{33}$ such that $\mathrm{supp}\,f_{3i}\subseteq E_{3}$ for $i=1,2,3$ and the set $\{P_{F_3}Tf_{3i}\}_{i=1}^3$ is linearly independent. 

Continue this process indefinitely. On the $n$-th step, use Lemma~\ref{sets-manipulation} to find $0<\varepsilon_{n-1}<\frac{\varepsilon_{n-2}}{2}$ such that the set $\{P_{Q_{n-1}^c\setminus A}TP_{Q_{n-1}\setminus B}f_{n-1,i}\}_{i=1}^{n-1}$ is linearly independent whenever $A$ and $B$ are such that $\mu(A)$, $\mu(B)<\varepsilon_{n-1}$. Fix disjoint sets $E_n$ and $F_n$ such that $\mu(E_n)$, $\mu(F_n)<\frac{\varepsilon_{n-1}}{2}$ and $\mathrm{rank}(P_{F_n}TP_{E_n})\ge n$; then there are $\{f_{ni}\}_{i=1}^n$ such that $\mathrm{supp}\,f_{ni}\subseteq E_n$ for all $i=1,2,\dots,n$, and the set $\{P_{F_n}Tf_{ni}\}_{i=1}^n$ is linearly independent. Finally, define $Q_n=(Q_{n-1}\cup E_n)\setminus F_n$. 

Observe that, since $E_n\cap F_n=\varnothing$ for all $n$, we have $Q_n=\cup_{k=1}^n\bigl(E_k\setminus(\cup_{i=k+1}^n F_i)\bigr)$, $n\in\mathbb N$. Put $Q=\cup_{k=1}^\infty\bigl(E_k\setminus(\cup_{i=k+1}^\infty F_i)\bigr)$. We will prove that $P_Q$ is a projection onto a half-space that is not $T$-almost-invariant.

Indeed, let us show that, for each $n\in\mathbb N$, the set $\{P_{Q^c}TP_Qf_{ni}\}_{i=1}^n$ is linearly independent.  Observe that one can write $Q=(Q_n\cup A_n)\setminus B_n$ where 
$A_n\subseteq\cup_{k=n+1}^\infty E_k$ 
and $B_n\subseteq\cup_{k=n+1}^\infty F_k$. In particular, $\mu(A_n)$, $\mu(B_n)<\sum_{k=n+1}^\infty\frac{\varepsilon_{k-1}}{2}$. The choice of $\varepsilon_{n}$ guarantees  $\varepsilon_{m}\le\frac{\varepsilon_{m-k}}{2^k}$ for all $1\le k\le m-1$ and $m\in\mathbb N$. Therefore $\mu(A_n)$, $\mu(B_n)<\varepsilon_n$. It follows that the set $\{P_{Q^c_n\setminus A_n}TP_{Q_n\setminus B_n}f_{ni}\}_{i=1}^n$ is linearly independent. Since $\mathrm{supp}\,f_{ni}\subseteq E_n\subseteq Q_n$ for all~$i$, we have $P_Qf_{ni}=P_{(Q_n\cup A_n)\setminus B_n}f_{ni}=P_{Q_n\setminus B_n}f_{ni}$. Hence, the inclusion  $Q_n^c\setminus A_n\subseteq Q^c$ implies that the set $\{P_{Q^c}TP_{Q}f_{ni}\}_{i=1}^n$ is linearly independent, too. It follows that the operator $P_{Q^c}TP_Q$ is of infinite rank. In particular, $P_Q$ is a half-space, and $P_Q$ is not $T$-almost-invariant.

\emph{Case 2}. $\mathcal H= L^2(X, \mu)$ where $\mu$ is an atomless $\sigma$-finite measure.

Assume that the set $\{\mathrm{rank}\, P_{E^c}TP_E\mid E\mbox{ measurable}\}$ is unbounded. First, let us show that for each $N\in\mathbb N$ and each set $A$ of finite measure, there exist disjoint subsets $E$ and $F$ of $A^c$ such that $\mu(E)$, $\mu(F)<\infty$ and $\mathrm{rank}(P_FTP_E)\ge N$.

Indeed, represent $T$ as $\left[\begin{smallmatrix}T_1 & T_2 \\ T_3 & T_4\end{smallmatrix}\right]$ relative to the decomposition $\mathcal H=P_A\oplus P_{A^c}$. By the assumptions of the theorem, $T_2$ and $T_3$ are of finite rank. Also, by \emph{Case 1}, there is $\kappa_A$ such that $\mathrm{rank}(P_{E^c}TP_E)\le\kappa_A$ for all measurable $E\subseteq A$. Pick a measurable set $G$ such that $\mathrm{rank}(P_{G^c}TP_G)\ge N+\mathrm{rank}T_2+\mathrm{rank}T_3+\kappa_A$. It follows from $P_{G^c}TP_G=P_{G^c\setminus A^c}TP_{G\setminus A}+P_{G^c\setminus A^c}TP_{G\setminus A^c}+P_{G^c\setminus A}TP_{G\setminus A}+P_{G^c\setminus A}TP_{G\setminus A^c}$ that $P_{G^c\setminus A}TP_{G\setminus A}$ has rank at least~$N$. Since both $G^c\setminus A$ and $G\setminus A$ can be approximated by sets of finite measure, the existence of $E$ and $F$ with required properties follows.

We will repeatedly use the above observation to construct a set $E_0$ such that\linebreak $\mathrm{rank}(P_{E_0^c}TP_{E_0})~=~\infty$. Pick disjoint sets $E_1$ and $F_1$ of finite measure such that $P_{E_1^c}TP_{E_1}\ne 0$. Next, pick disjoint sets $E_2$, $F_2\subseteq (E_1\cup F_1)^c$ of finite measure such that $\mathrm{rank}(P_{F_2}TP_{E_2})\ge 2$. Continuing inductively, we construct two sequences $(E_n)$ and $(F_n)$ of sets of finite measure such that every member of either sequence is disjoint from all other members of both sequences and $\mathrm{rank}(P_{F_n}TP_{E_n})\ge n$ for all $n\in\mathbb N$. Define $E_0=\cup_{n\in\mathbb N}E_n$. Then $\mathrm{rank}(P_{E_0^c}TP_{E_0})\ge\mathrm{rank}(P_{F_n}TP_{E_n})\ge n$ for all $n\in\mathbb N$, so that $P_{E_0}$ is a projection onto a halfspace that is not $T$-almost-invariant.

\emph{Case 3}. $\mathcal H= L^2(X, \mu)$ where $\mu$ is a purely atomic measure.

This case is proved by the same argument as \emph{Case 2}. We simply approximate the measurable sets by finite sets instead of sets of finite measure, then repeat the argument almost verbatim.
\end{proof}


\bigskip
Before proceeding to the proof of the main result of this section, we pause for two more technical lemmas.  The thrust of the second of these  lemmas is to show that if, in the setting of a continuous measure space, an ``off-diagonal  corner" of an operator $T$ has rank $\kappa$, then we can compress that corner into as small a ``subcorner" as we wish (in the sense that both the range and initial space correspond to sets of small measure) and still keep the rank of the compression as great as the rank of the original  corner.

\bigskip

\begin{lem} \label{lem4.41}
Let $\kappa \ge 1$ be an integer.  Suppose that $Y$ is a subset of a continuous, $\sigma$-finite, Borel measure space $(X, \mu)$, and that $\{ f_1, f_2, ..., f_\kappa\}$ is a linearly independent family of measurable functions from $Y$ to $\bbC$.   Then for all $\varepsilon > 0$ there exists $W \subseteq Y$ with $0 < \mu(W) < \varepsilon$ such that  $\{{f_1}|_W, {f_2}|_W, ..., {f_\kappa}|_W\}$ is also a linearly independent set.
\end{lem}

\begin{pf}
The proof will be by induction on $\kappa$.  The case $\kappa = 1$ is clear, since any non-zero function $f_1$ must still be non-zero on a set of arbitrarily small measure.  (The measure is continuous!)

Assume that the result holds for $\kappa - 1$.   We shall prove that it holds for $\kappa$.   Indeed, suppose otherwise.  Then there exists $\varepsilon > 0$ such that for all $W \subseteq Y$ with $\mu(W) < \varepsilon$, the set 
\[
\{ {f_1}|_W, {f_2}|_W, ..., {f_\kappa}|_W \} \]
is linearly dependent.   Use the induction hypothesis to choose $W_0 \subseteq Y$ with $\mu (W_0) < \frac{\varepsilon}{2}$ such that $\{ {f_1}|_{W_0}, {f_2}|_{W_0}, ..., {f_{\kappa-1}}|_{W_0} \}$ is linearly independent.  Since $\{ {f_1}|_{W_0}, {f_2}|_{W_0}, ..., {f_\kappa}|_{W_0} \}$ is linearly dependent, there exist unique  $\lambda_1, \lambda_2, ..., \lambda_{\kappa-1} \in \bbC$ so that 
\[
{f_\kappa}|_{W_0} = \lambda_1 {f_1}|_{W_0} + \lambda_2 {f_2}|_{W_0} + \cdots + \lambda_{\kappa - 1} {f_{\kappa-1}}|_{W_0}. \]
If $W_1 \subseteq Y$ and $\mu(W_1) \le \frac{\varepsilon}{2}$, then 
\[
\{ {f_1}|_{W_0 \cup W_1}, {f_2}|_{W_0 \cup W_1}, ..., {f_\kappa}|_{W_0 \cup W_1} \} \]
is linearly dependent.   Thus some linear combination of these functions must be zero.   By restricting to the subset $W_0$ of $W_0 \cup W_1$, we see that the only possible linear combination which does this is:
\[
\lambda_1 {f_1}|_{W_0 \cup W_1} + \lambda_2 {f_2}|_{W_0 \cup W_1} + \cdots + \lambda_{\kappa - 1} {f_{\kappa-1}}|_{W_0 \cup W_1} - {f_{\kappa}}|_{W_0 \cup W_1} = 0. \]

Now, since $\mu$ is $\sigma$-finite, we can write $Y = \cup_\alpha W_\alpha$ with $\mu(W_\alpha) < \frac{\varepsilon}{2}$ for all $\alpha$.   But then 
\[
\lambda_1 {f_1}|_{Y} + \lambda_2 {f_2}|_{Y} + \cdots + \lambda_{\kappa - 1} {f_{\kappa-1}}|_{Y} - {f_\kappa}|_{Y} = 0, \]
a contradiction.
\end{pf}


\begin{lem} \label{lem4.42}
Let   $\mu$ be a continuous, $\sigma$-finite Borel measure on a measure space $X$.  Suppose that $\cD = \{ M_f: f \in L^\infty(X, \mu)\}$ is a masa acting on $\hilb = L^2(X, \mu)$, $Z_0 \subseteq X$ is a measurable set, and $P_0 :=P_{Z_0} \in \cD$   is a projection for which 
\[
\mathrm{rank}\, (P_0^\perp T P_0) = \kappa < \infty. \]
Then for each $m \ge 1$, there exist $Y_m, Z_m \subseteq X$ measurable subsets satisfying
\begin{enumerate}
	\item[(i)] 
	$\max\{ \mu(Y_m), \mu(Z_m)\} < \frac{1}{m}$;
	\item[(ii)]   
	$Y_{m+1} \subseteq Y_m \subseteq Z_0^c$, $Z_{m+1} \subseteq Z_m \subseteq Z_0$ for all $m \ge 1$;
\end{enumerate}
for which 
\[
\mathrm{rank}\, ( P_{Y_m} T P_{Z_m}) = \kappa. \]
\end{lem}

\begin{pf}
Since $\mathrm{rank} (P_0^\perp T P_0) = \kappa$, there exist linearly independent functions $f_1, f_2,..., f_\kappa: Z_0^c \to \bbC$ and linearly independent functions $g_1, g_2,..., g_\kappa: Z_0 \to \bbC$ such that 
\[
P_0^\perp T P_0 = \sum_{j=1}^\kappa f_j \otimes g_j^*. \]
By Lemma~\ref{lem4.41}, we can find $Y_1 \subseteq Z_0^c$ and $Z_1 \subseteq Z_0$ measurable so that\linebreak $\max \{ \mu(Y_1), \mu(Z_1) \} < 1$ and both $\{ {f_1}|_{Y_1}, {f_2}|_{Y_1}, ..., {f_\kappa}|_{Y_1} \}$ and $\{ {g_1}|_{Z_1}, {g_2}|_{Z_1}, ..., {g_\kappa}|_{Z_1} \}$ are linearly independent.   Then 
\[
P_{z_1} T P_{Y_1} = \sum_{j=1}^\kappa {f_j}|_{Y_1} \otimes {{g_j}|_{Z_1}}^* \]
has rank $\kappa$.  The result then follows from a routine induction argument.
\end{pf}


The following is the main result of this section.


\begin{thm} \label{thm4.5} 
Let $\hilb$ be a separable Hilbert space and let $\cD$ be a masa in $\bofh$.  Suppose that $T \in \bofh$ has the property that the range of every projection in $\cD$ is almost-invariant for $T$.   Then there exists $F \in \cF(\hilb)$ and $D \in \cD$ so that $T = D + F$.   Furthermore, if $\kappa=\sup \{ \mathrm{rank} \, (I-P)T P : P \in \cP(\cD)\}$ then $\kappa<\infty$ and
\begin{enumerate}
	\item[(a)]
	If $\cD$ is an atomic masa, then $F$ may be chosen so that $\mathrm{rank} \, F \le 3 \kappa$.
	\item[(b)]
	If $\cD$ is a continuous masa,  then $F$ is uniquely determined and
	$\mathrm{rank} \, F \le \kappa$.
	\item[(c)]
	For $\cD$ a general masa,  $F$ may be chosen so that
	$\mathrm{rank} \, F \le 6\kappa$.
\end{enumerate}	
\end{thm}

\begin{pf}
First observe that by Theorem~\ref{masa-bounded-defect},  $\kappa=\sup \{ \mathrm{rank} \, (I-P)T P : P \in \cP(\cD)\}<\infty$ and
\[
\kappa  = \max  \{ \mathrm{rank} \, (I-P)T P : P \in \cP(\cD)\}. \]
\smallskip

\noindent(a)\ \ \ 
	We begin with the case where $\cD$ is a discrete masa.   Let $\cE := \{ \xi_\alpha: \alpha \in A\}$ denote the orthonormal basis for $\hilb$ corresponding to characteristic functions onto the atoms $A$ of the masa.   Choose $P_0 \in \cD$ so that $\kappa = \mathrm{rank}\, (I-P_0) T P_0$.   It follows that the matrix of $(I-P_0) T P_0$ relative to $\{ \xi_\alpha\}_\alpha$ admits $\kappa$ linearly independent columns.  We shall relabel the corresponding basis vectors as   $\{ e_1, e_2, ..., e_\kappa\}$.    In a similar way, we may find $\kappa$ linearly independent rows of $(I-P_0) T P_0$, and we may relabel the basis vectors corresponding to those rows as $\{ f_1, f_2, ..., f_\kappa\}$.   (Observe that since $P_0$ is obviously perpendicular to $(I-P_0)$, it follows that $\{ e_1, e_2, .., e_\kappa\} \cap \{ f_1, f_2, ..., f_\kappa\} = \varnothing$.)  
	
	Let us define $\hilb_e = \mathrm{span} \{ e_j\}_{j=1}^\kappa$, $\hilb_f = \mathrm{span} \{ f_j\}_{j=1}^ \kappa$, and $\hilb_1 = \hilb \ominus (\hilb_e \oplus \hilb_f)$.
	We may then write the matrix for $T$ relative to $\hilb = \hilb_e \oplus \hilb_1 \oplus \hilb_f$ as
	\[
	T = \begin{bmatrix} T_1 & T_2 & T_3 \\ C & B & T_6 \\ T_0 & A & T_9 \end{bmatrix}. \]
	Let $Q$ denote the orthogonal projection of $\hilb$ onto $\hilb_e \oplus \hilb_1$.  Then 
	\[
	\kappa \ge \mathrm{rank} \,(Q^\perp T Q) = \mathrm{rank} \begin{bmatrix} T_0 & A \end{bmatrix} \ge \mathrm{rank}\, T_0 = \kappa.\]
	It follows that $A = T_0 W$ for some $W \in \cB(\hilb_1, \hilb_f)$.  Similarly, $C = X T_0$ for some $X \in \cB(\hilb_e, \hilb_1)$.
	
	Let $g \in \cE \backslash \{ e_1, e_2, ..., e_\kappa, f_1, f_2, ..., f_\kappa\}$, and let $R$ denote the orthogonal projection of $\hilb$ onto $\mathrm{span}\, \{ g, f_1, ..., f_\kappa\}^\perp$.  Then $\mathrm{rank} (R^\perp T R) = \kappa$, since $T_0$ is a compression of $R^\perp T R$.
	Let $\hilb_2 =  \hilb_{1} \ominus \bbC g$.  Relative to the decomposition $\hilb = \hilb_e \oplus  \hilb_2 \oplus  \bbC g \oplus \hilb_f$, we may write
	\[
	T = \begin{bmatrix} T_1 & T_{2, 1} & T_{2, g} & T_3 \\ C_1 & B_{2, 2} & B_{2,g} & T_{6,1} \\ C_2 & B_{g, 2} & B_{g,g} &  T_{6, g} \\ T_0 & A_1 & A_g & T_9 \end{bmatrix}. \]
	Here, $B_{g, g} = \langle T g, g \rangle $.
	Thus $\mathrm{rank}\, \begin{bmatrix} C_2 & B_{g, 2} \\ T_0 & A_1 \end{bmatrix} = \kappa = \mathrm{rank} \begin{bmatrix} T_0 & A_1  & A_g \end{bmatrix}.$
	
	In particular, this means that the column  of $A_g$ must be a linear combination of the columns of $T_0$.   We can therefore choose a new entry $z_g \in \bbC$ so that the column $\begin{bmatrix} z_g \\ A_g \end{bmatrix}$ is precisely the same linear combination of the columns of $\begin{bmatrix} C_2 \\ T_0\end{bmatrix}$. In fact, by Lemma~\ref{lem4.43aBanach}, $z_g=C_2T_0^{-1}A_g$, so that $\left|z_g\right|\le\|C_2\|\cdot\|T_0^{-1}\|\cdot\|A_g\|\le\|T\|^2\cdot\|T_0^{-1}\|$. In particular, the diagonal operator $D_1 = \mathrm{diag} (B_{g, g} - z_g): g \in \cE \backslash \{ e_1, e_2, ..., e_\kappa, f_1, f_2, ..., f_\kappa\}$ is bounded.  Set $D = 0 \oplus D_1 \oplus 0 \in \cB(\hilb_e \oplus \hilb_1 \oplus \hilb_f)$, so that $D \in \cD$.   It then follows that with  \[
	T - D = \begin{bmatrix} T_1 & T_2 & T_3 \\ C & B - D_1 & T_6 \\ T_0 & A & T_9 \end{bmatrix}, \]
	every row  of $\begin{bmatrix} C & B-D_1 \end{bmatrix}$ is a linear combination of the rows of $\begin{bmatrix} T_0 & A\end{bmatrix}$, so that 
	\[
	\mathrm{rank} \begin{bmatrix} C & B-D_1 \\ T_0 & A \end{bmatrix} = \kappa.\]
	Finally, 
	\[
	\mathrm{rank}\, (T-D) \le \mathrm{rank} \begin{bmatrix} C & B-D_1 \\ T_0 & A \end{bmatrix} 
	+ \mathrm{rank}\, \begin{bmatrix} T_1 & T_2 & T_3 \end{bmatrix} 
	+ \mathrm{rank}\, \begin{bmatrix} 0 \\ T_6 \\T_9 \end{bmatrix} \le 3 \kappa. 
	\]
	(That the ranks of the last two operators are each at most $\kappa$ follows from the fact that the range of the second operator is contained in $\mathrm{span}\, \{e_1, e_2, ..., e_\kappa\}$, while the domain of the third operator is $\mathrm{span}\, \{ f_1, f_2, ..., f_\kappa\}$.)
\smallskip	

\noindent(b) \ \ \ 
	Next we consider the case where the underlying measure space is atomless.
	Let $Z_0\subseteq X$ be measurable and such that $P_{Z_0}\in\cD$ satisfies $\mathrm{rank}(P_{Z_0}^\perp TP_{Z_0})=\kappa$. Write $P_{Z_0}^\perp TP_{Z_0}=\sum_{j=1}^{\kappa}f_j\otimes g_j^*$ for some $f_j\in L^2(Z_0^c,\mu)$, $g_j\in L^2(Z_0,\mu)$, $1\le j\le\kappa$.
	
	Consider $Y_n$, $Z_n$, $n \ge 1$ chosen as in Lemma~\ref{lem4.42}.   Then 
	\[
	\mathrm{rank} (P_{Z_n}^\perp T P_{Z_n}) \le \kappa = \sup_{P \in \cP(\cD)}  \mathrm{rank} (P^\perp T P) \]
	for all $n \ge 1$.   On the other hand, $P_{Y_n} P_{Z_n} = 0$ implies that 
	\[
	\kappa = \mathrm{rank}\, (P_{Y_n} T P_{Z_n}) \le \mathrm{rank}\, (P_{Z_n}^\perp T P_{Z_n}) \le \kappa, \]
	and hence $\mathrm{rank}\, (P_{Z_n}^\perp T P_{Z_n}) = \kappa$.
	
	This allows us to write $P_{Z_n}^\perp T P_{Z_n} = \sum_{j=1}^\kappa u_j \otimes v_j^*$ where $u_j: Z_n^c \to \bbC$, $v_j: Z_n \to \bbC$, $1 \le j \le \kappa$ and each of $\{ u_j\}_{j=1}^\kappa$ and $\{ v_j\}_{j=1}^\kappa$ is linearly independent.  But then 
	\[
	P_{Y_n} (P_{Z_n}^\perp T P_{Z_n}) = P_{Y_n} T P_{Z_n} = \sum_{j=1}^\kappa f_j|_{Y_n} \otimes (g_j|_{Z_n})^* = \sum_{j=1}^\kappa u_j \otimes v_j^*. \]
	From this it follows that $\mathrm{span}\, \{ v_1, v_2, ..., v_\kappa\} = \mathrm{span}\, \{ {g_1}|_{Z_n}, {g_2}|_{Z_n}, ..., {g_\kappa}|_{Z_n} \}$.  Hence we may rewrite 
	\[
	P_{Z_n}^\perp T P_{Z_n} = \sum_{j=1}^\kappa \overline{u}_j^{(n)} \otimes (g_j|_{Z_n})^*\]
	for some choice of $\overline{u}_j^{(n)}: {Z_n}^c \to \bbC$, $\{ \overline{u}_j^{(n)}\}_{j=1}^\kappa$ linearly independent.  But then 
	\[
	\sum_{j=1}^\kappa \overline{u}_j^{(n)}|_{Y_n} \otimes ({g_j}|_{Z_n})^* = \sum_{j=1}^\kappa {f_j}|_{Y_n} \otimes ({g_j}|_{Z_n})^* \]
	implies that $\overline{u}_j^{(n)}|_{Y_n} = {f_j}|_{Y_n}$.
	
	Now, for $m \ge n$, we may write
	\[
	P_{Z_m}^\perp T P_{Z_m} = \sum_{j=1}^\kappa \overline{u}_j^{(m)} \otimes ({g_j}|_{Z_m})^*. \]
	Also, for $m \ge n$, observe that $P_{Z_n}^\perp (P_{Z_m}^\perp T P_{Z_m}) P_{Z_m} = P_{Z_n}^\perp (P_{Z_n}^\perp T P_{Z_n}) P_{Z_m}$, which implies that
	\[
	\sum_{j=1}^\kappa  \overline{u}_j^{(m)}|_{ {Z_n}^c} \otimes ({g_j}|_{Z_m})^*
	= \sum_{j=1}^\kappa \overline{u}_j^{(n)}|_{Y_n} \otimes ({g_j}|_{Z_m})^*, \]
	which in turn implies that $\overline{u}_j^{(m)}|_{ {Z_n}^c} = \overline{u}_j^{(n)}|_{Y_n}$.   This allows us to define $\overline{u}_j: X \to \bbC$ by 
	\[
	x \mapsto \lim_{n\to \infty} \overline{u}_j^{(n)} (x), \]
	which is well-defined on $X \backslash \cap_{n=1}^\infty Z_n$.  Note that $\cap_{n=1}^\infty Z_n$ has measure zero, so that $\overline{u}_j$ is defined almost everywhere on $X$. 
	
	Note that for any $n \ge 1$, 
	\[
	P_{Z_n}^\perp T P_{Z_n} = \sum_{j=1}^\kappa {\overline{u}_j}|_{ {Z_n}^c} \otimes ( {g_j}|_{ {Z_n}})^*. \]
	In a similar manner, by considering $P_{Y_n}  T P_{Y_n}^\perp$, we can construct functions $\overline{v}_1, \overline{v}_2, ..., \overline{v}_\kappa : X \to \bbC$ so that 
	\[
	P_{Y_n}  T P_{Y_n}^\perp = \sum_{j=1}^\kappa {f_j}|_{P_{Y_n} X} \otimes ( {v_j}|_{P_{Y_n}^\perp X})^* \mbox{ \ \ \ \ \ \ \ for all } n \ge  1. \]
	Now consider $F = \sum_{j=1}^\kappa \overline{u}_j \otimes \overline{v}_j^*$.     We claim that both $\{ \overline{u}_1, \overline{u}_2, ..., \overline{u}_\kappa \}$ and $\{ \overline{v}_1, \overline{v}_2, ..., \overline{v}_\kappa\}$ are linearly independent sets.   
Indeed, since $\{\overline{u}_j|_{Y_n} \}_{j=1}^\kappa = \{ f_j|_{Y_n}\}_{j=1}^\kappa$ is linearly independent, so is $\{ \overline{u}_1, \overline{u}_2, ..., \overline{u}_\kappa \}$;  similarly  $\{ \overline{v}_1, \overline{v}_2, ..., \overline{v}_\kappa\}$ is linearly independent.
 Hence $\mathrm{rank}\, F = \kappa$.  
	
	Let $D = T - F$.   We shall prove that $D \in \cD$ by showing that $D$ commutes with every projection in $\cD$.   
	
	Let $R \in \cP(\cD)$ be a projection satisfying $P_{Z_n} \le R \le P_{Y_n^c}$. That is, $R=P_V$ for some measurable set $V$ satisfying $Z_n\subseteq V\subseteq Y_n^c$. We have $\mathrm{rank}\, (R^\perp T R) \le \kappa$. 
	Consider $R^\perp T R$ as an operator from $R\cH$ to $R^\perp\cH$.   Relative to the decomposition $R\cH = P_{Z_n}\cH \oplus (P_{V\setminus Z_n}\cH)$ and $R^\perp\cH = (P_{V^c\setminus Y_n}\cH) \oplus P_{Y_n}\cH$, we may write 
	\begin{align*}
		R^\perp (P_{Z_n}^\perp T P_{Z_n})|_{R\cH} &= \begin{bmatrix} A & 0 \\ C & 0 \end{bmatrix}; \\
		(P_{Y_n} T P_{Y_n}^\perp)|_{R\cH} &= \begin{bmatrix} 0 & 0 \\ C & B \end{bmatrix}; \\
		P_{Y_n} T P_{Z_n}|_{R\cH} &= \begin{bmatrix} 0 & 0 \\ C & 0 \end{bmatrix}; 
		\mbox{ \ \ \ and } \\ 
		P_{V^c \setminus Y_n} T P_{V\setminus Z_n}|_{R\cH} &= \begin{bmatrix} 0 & X \\ 0 & 0 \end{bmatrix},
	\end{align*}	
	so that 
	\begin{align*}
	(R^\perp T R)|_{R\cH}
		&= \big(R^\perp (P_{Z_n}^\perp T P_{Z_n}) + (P_{Y_n} T P_{Y_n}^\perp) R - P_{Y_n} T P_{Z_n} + P_{V^c \setminus Y_n} T P_{V\setminus Z_n}\big)|_{R\cH} \\
		&= \begin{bmatrix} A & X \\ C & B \end{bmatrix}.
	\end{align*}	
	Since $\mathrm{rank}\, (T^\perp T R)|_{R \hilb} = \kappa = \mathrm{rank}\, C = \mathrm{rank}\, (P_{Y_n} T P_{Z_n})|_{R \hilb}$, 
	by Lemma~\ref{lem4.43aBanach}, $X$ is uniquely determined by $A, B$ and $C$.	
	Observe that
	\[
		R^\perp P_{Z_n}^\perp T P_{Z_n} = R^\perp (\sum_{j=1}^\kappa {\overline{u}_j}|_{Z_n^c} \otimes ({\overline{v}_j}|_{Z_n}^* ),\]
	\[
  P_{Y_n} T P_{Y_n}^\perp = (\sum_{j=1}^\kappa {\overline{u}_j}|_{Y_n} \otimes ({\overline{v}_j}|_{Y_n^c})^* )R, \quad\mbox{and}\]
	\[
  P_{Y_n} T P_{Z_n} = (\sum_{j=1}^\kappa {\overline{u}_j}|_{Y_n} \otimes ({\overline{v}_j}|_{Z_n})^* )R. \]
		
	But  $X = \sum_{j=1}^\kappa {\overline{u}_j}|_{V^c\setminus Y_n} \otimes ({\overline{v}_j}|_{V\setminus Z_n})^*$ is a particular choice which satisfies 
	\[
	\mathrm{rank}\, \begin{bmatrix} A & X \\ C & B \end{bmatrix} = \mathrm{rank}\, C = \kappa, \]
	for 
	\[
	R^\perp T R = \sum_{j=1}^\kappa {\overline{u}_j}|_{R^\perp} \otimes ({\overline{v}_j}|_{R})^* \]
	has rank $\kappa$ by virtue of the fact that $\{ {\overline{u}_j}|_{Y_n} \}_{j=1}^\kappa, \{ {\overline{v}_j}|_{Z_n} \}_{j=1}^\kappa$ are linearly independent families.
	
	Thus $R^\perp T R = R^\perp F R$ for any projection $R \in  \cD $ with $P_{Z_n} \le R \le P_{Y_n}$.
	
	Since $\mu (\cap_{n=1}^\infty Z_n) = 0 = \mu(\cap_{n=1}^\infty Y_n)$, it follows that $Q^\perp T Q = Q^\perp F Q$ for all projections $Q \in \cD$.
	But then
	\begin{align*}
	Q D - D Q 
		&=	Q (T - F) - (T - F) Q \\
		&= (Q T - T Q) - (Q F - F Q) \\
		&= (Q T Q^\perp + Q T Q) - (Q^\perp T Q + Q T Q) - (Q F - FQ) \\
		&= Q F Q^\perp - Q^\perp F Q - Q F + F Q \\
		&= - Q F Q + Q F Q \\
		&= 0 
	\end{align*}
	for all projections $Q \in \cD$, whence $D = T - F \in \cD$, completing the proof in this setting.
	
	Finally, let us show that the decomposition $T=D+F$ is unique. Let $D_1,D_2\in\cD$ and $F_1,F_2\in\cF(\cH)$ be such that $T=D_1+F_1=D_2+F_2$. Then $F_1-F_2=D_2-D_1\in\cD$. Since all non-zero multiplication operators have infinite rank, we have $F_1=F_2$. But then $D_1=T-F_1=T-F_2=D_2$.
\smallskip
	
\noindent(c)\ \ \ 
	Now consider the case of an arbitrary, $\sigma$-finite measure space $(X, \mu)$.   We first partition the space into its discrete and continuous parts $(X_d, \mu_d)$ and $(X_c, \mu_c)$ (where $\mu_d = \mu|_{X_d}$ and $\mu_c = \mu|_{X_c}$ respectively) and consider the corresponding decomposition of the Hilbert space:
	\[
	\hilb = L^2(X, \mu) = L^2(X_d, \mu_d) \oplus L^2(X_c, \mu_c). \]
	Let $P_d$ (resp. $P_c$) denote the orthogonal projection of $\hilb$ onto $\hilb$, and note that $P_d, P_c \in \cD$.  As such, this  produces a  decomposition of the masa $\cD = \cD_d \oplus \cD_c$.   Relative to the above decomposition of $\hilb$, let us write $T = \begin{bmatrix} T_1 & T_2 \\ T_3 & T_4 \end{bmatrix}$.   It is not too difficult to see that the condition that $\sup_{P \in \cP(\cD)} \mathrm{rank}\, (P^\perp T P) = \kappa$ implies that 
	\[
	\sup_{P \in \cP(\cD_d)} P_d (P^\perp T_1 P) P_d \le \kappa, \]
	and 
	\[
	\sup_{P   \in \cP(\cD_c)} P_c (P^\perp T_4 P) P_c \le \kappa. \]
	Also, the estimates $\mathrm{rank}\, T_2$, $\mathrm{rank}\, T_3 \le \kappa$ are clear.   

	From part (a), we may write  $T_1 = D_1 + F_1$ with $T_1 \in \cD_d$, $F_1 \in \cB(L^2(X_d, \mu_d))$ and $\mathrm{rank}\, F_1 \le 3 \kappa$.  In a similar way, by part (b), there exist $D_4 \in \cD_c$, $F_4 \in \cB(L^2(X_c, \mu_c))$ with $\mathrm{rank}\, F_4 \le \kappa$ so that $T_4 = D_4 + F_4$.
	\smallskip
	
	Finally, 	$T = D + F$, where $D = D_1 \oplus D_4 \in \cD$, and $F = \begin{bmatrix} F_1 & T_2 \\ T_3 & F_4 \end{bmatrix}$ satisfies  $\mathrm{rank} F \le  \mathrm{rank}\, F_1 + \mathrm{rank}\, T_2 +  \mathrm{rank}\, T_3 +  \mathrm{rank}\, F_4 \le (3 + 1 + 1 + 1) \kappa = 6 \kappa$. 
\end{pf}

\begin{rem} It should be clear that in the case of a discrete masa, no uniqueness of the decomposition $T=D+F$ can be expected. Indeed, let, for example, $\cD=\ell_\infty$ be a masa in $\cB(\ell_2)$ and $T\in\cB(\cH)$ be arbitrary. By Theorem~\ref{thm4.5} we can find a decomposition $T=D+F$ with $D\in\cD$, $F\in\cF(\cH)$. Put $D_1=D+e_1\otimes e_1^*$ and $F_1=F-e_1\otimes e_1^*$, where $e_1=(1,0,0,\dots)$ is the first basic vector. It is clear that $D_1\in\cD$, $F_1\in\cF(\cH)$, and $T=D_1+F_1$.
\end{rem}

\bigskip

The following theorem is a version of \cite[Theorem~2.7]{Popov2010} in the context of masas.

\begin{thm}  
Suppose that $\cA$ is a norm-closed algebra and that $\cD \subseteq \cB(L^2(X, \mu))$ is a masa.   Suppose also that for all projections $P \in \cD$ and all $A \in \cA$ we have that 
\[
\mathrm{rank}\, (A P - P A P) < \infty. \]
Then there exists $\kappa > 0$ so that $\mathrm{rank}\, (A P - P A P) \le \kappa$ for all $A \in \cA$ and  $P \in \cP(\cD)$.
\end{thm}

\begin{pf}
Let $\cE_k := \{ T \in \cA:  \mathrm{rank}\, (I-P)T P \le k \mbox{ for all } P \in \cP(\cD) \}$.   Since every member in $\cA$ is of the form $D + F$ for some $D\in \cD$ and $F\in\cF(\hilb)$ by Theorem~\ref{thm4.5}, it follows that 
\[
\cA = \cup_{k=1}^\infty \cE_k. \]

We now claim that $\cE_k$ is closed for each $k \ge 1$.   To that end, let $P \in \cP(\cD)$.   If $T_n = D_n + F_n \in \cE_k$ and $\lim_n T_n = T \in \cA$, then 
\begin{align*}
	(I-P) T P 
		&= \lim_n (I-P) T_n P \\
		&= \lim_n (I-P) (D_n+F_n) P \\
		&= \lim_n (I-P)F_n P.
\end{align*}
But $\mathrm{rank}\, (I-P)F_nP \le k$ for all $n \ge 1$ implies by the lower-semicontinuity of the rank that $\mathrm{rank}\, (I-P)TP \le k$, whence $T \in \cE_k$ and $\cE_k$ is closed.

By the Baire Category Theorem, we get that the interior  of $\cE_k$ is non-empty for some $k > 0$.   Choose such a $k$, and suppose that $T_0 = D_0 + F_0$ lie in the interior of $\cE_k$.	Then $\cE_k - T_0 := \{ T - T_0: T \in \cE_k \}$ contains a ball in $\cA$ of positive radius centred at $0$.   But then for all $A \in \cA$, there exists $t > 0$ so that $t A \in \cE_k - T_0$.   That is, $t A = T_1 - T_0$ for some $T_1 \in \cE_k$.   

Hence  for all projections $P \in \cD$, 
\begin{align*}
	\mathrm{rank}(I-P) A P 
		&= \mathrm{rank} (I-P) tA P \\
		&= \mathrm{rank} (I-P) (T_1 - T_0) P \\
		&\le \mathrm{rank} (I-P) T_1 P + \mathrm{rank} (I-P)T_0 P \\
		&\le k + k. 
\end{align*}
			
Setting $\kappa = 2 k$ completes the proof.			
			
\end{pf}		
			

\begin{cor}\label{cor-4.10} Let $\cA\subseteq\cB(\cH)$ be a norm-closed algebra and $\cD\subseteq\cB(\cH)$ be a masa such that $(I-P)TP\in\cF(\cH)$ for all $T\in\cA$ and  $P\in \cP(\cD)$. Then there exists $k\in\bbN$ such that $\cA\subseteq\cD+\cF_k(\cH)$ where $\cF_k(\cH)=\{T\in\cB(\cH)\mid\mathrm{rank}\,T\le k\}$.
\end{cor}
		

\bigskip

In the case of a continuous measure space, we are able to say more about the structure of an algebra $\cA$ satisfying the condition above. We need an auxiliary lemma.

\begin{lem}\label{results-of-radjavi}
Suppose that $\cL\subseteq\cF_k(\hilb)$ is a linear subspace. Then there exist two invertible operators $S$ and $T$ in $\cB(\hilb)$ and a projection $P=P^*=P^2\in\cB(\hilb)$ of rank at most $k$ such that $S\cL T\subseteq P\cB(\hilb)+\cB(\hilb)P$.
\end{lem}
\begin{proof}
Pick an operator $A_0\in\cL$ of maximal rank. Let $\kappa=\mathrm{rank}\,(A_0)$. There exist two invertible operators $S$ and $T$ such that $SA_0T=\begin{bmatrix} I_\kappa & 0 \\ 0 & 0 \end{bmatrix}$. We claim that $P=\begin{bmatrix} I_\kappa & 0 \\ 0 & 0 \end{bmatrix}$ satisfies the conclusion of the lemma.

Indeed, let $A\in\cL$ be arbitrary. Write $SAT=\begin{bmatrix} A_{11} & A_{12} \\ A_{21} & A_{22} \end{bmatrix}$. It is enough to prove that $A_{22}=0$. For every scalar $\lambda$, we have $S(A+\lambda A_0)T=\begin{bmatrix} \lambda I_\kappa+A_{11} & A_{12} \\ A_{21} & A_{22} \end{bmatrix}$. The matrix $\lambda I_\kappa+A_{11}$ is invertible for all large enough~$\lambda$, so that $\mathrm{rank}\,\begin{bmatrix} \lambda I_\kappa+A_{11} & A_{12} \\ A_{21} & A_{22} \end{bmatrix}=\mathrm{rank}\,(\lambda I_\kappa+A_{11})$ for large~$\lambda$. By the reasoning analogous to that of Lemma~\ref{lem4.43aBanach}, we get $A_{22}=A_{21}(\lambda I_\kappa+A_{11})^{-1}A_{12}$. Taking $\lambda\to\infty$, we find that $A_{22}=0$.
\end{proof}

\begin{thm} \label{thm4.8}
Let $\mu$ be a continuous measure on a measure space $X$.  Suppose that $\hilb=L^2(X, \mu)$ and $\cA \subseteq \cB(\hilb)$ is a closed algebra such that $\mathrm{rank}\, (P^\perp T P) \le \kappa $ for some constant $\kappa$ and for all projections $P \in \cD$ and all $T \in \cA$.   Then there exist projections $Q, R \in\bofh$ so that 
\[
\cA \subseteq \cD + Q \bofh +  \bofh R, \]
with $\mathrm{rank}\, Q \le \kappa$ and $ \mathrm{rank}\, R \le \kappa$.
\end{thm}

\smallskip

\noindent{\textbf{Remark.}} In view of Corollary~\ref{cor-4.10} the existence of $\kappa$ from the hypothesis of Theorem~\ref{thm4.8} follows from the condition $\mathrm{rank}\, (P^\perp T P)<\infty$ for all projections $P \in \cD$ and all $T \in \cA$.

\smallskip

\begin{pf}
By Theorem~\ref{thm4.5} (b), each $T \in \cA$ may be written in a  unique  way as $T  = D_T + F_T$, where $D_T \in \cD$, and $\mathrm{rank}\, F_T < \infty$.   Since $\cA$ is a vector space, so is the set $\cF_\cA := \{ F_T : T \in \cA\}$.   Indeed, for $T_1, T_2 \in \cA$ and $\lambda \in \bbC$, 
	\[
	\lambda_1 T_1 + T_2 = (\lambda D_{T_1} + D_{T_2}) + (\lambda F_{T_1} + F_{T_2}) \]
	is a particular decomposition of $\lambda T_1 + T_2$ as the sum of an element of the masa with a finite rank operator.   Since the decomposition is unique, it follows that $D_{\lambda_1 T_1 + T_2} = \lambda_1 D_{T_1} + D_{T_2}$, and similarly, $F_{\lambda_1 T_1 + T_2} = \lambda_1 F_{T_1} + F_{T_2}$.    	
	Moreover, as we saw in Theorem~\ref{thm4.5}(b), $\mathrm{rank}\, F_T \le \kappa$ for all $T \in \cA$, and so $\cF_\cA$ is a vector space of operators having rank at most $\kappa$. 
	
	By Lemma~\ref{results-of-radjavi}, there exist invertible operators $S,T\in\cB(\hilb)$ and a projection of rank at most $k$ such that $\cF_\cA\subseteq S^{-1}P\cB(\hilb)T^{-1}+S^{-1}\cB(\hilb)PT^{-1}\subseteq S^{-1}PS\cB(\hilb)+\cB(\hilb)TPT^{-1}$. Now, if $Q$ is the orthogonal projection onto the range of $S^{-1}PS$ and $R$ is the orthogonal projection onto $\ker(TPT^{-1})^\perp$ then $\mathrm{rank}\, Q$, $\mathrm{rank}\, R \le \kappa$ and $\cF_\cA \subseteq Q \bofh + \bofh R$.   
\end{pf}


\begin{eg} \label{eg4.9}
While Theorem~\ref{thm4.8} guarantees the existence of projections $Q, R \in \bofh$ so that $A \subseteq \cD + Q \bofh + \bofh R$, there may not exist $Q, R \in \cD$ so that $A \subseteq \cD + Q \bofh + \bofh R$.  
To see this, let $\hilb = L^2([0,1], dx)$ where ``dx" denotes Lebesgue measure and let $F \in \bofh$ be a rank one operator without invariant subspaces of the form $L^2(X, dx)$, where $X\subseteq[0,1]$ measurable.   For example, we may let $f: [0,1] \to \bbC$ be the constant function $f(x) = 1$ for all $x \in [0,1]$, and set $F = f \otimes f^*$.   Set $\cA = \bbC I + \bbC F$.   Let $\cD$ denote the canonical masa $\cD \simeq L^\infty([0,1], dx)$ in $\bofh$.  Then $\cA$ is a norm-closed algebra and clearly $\cA \subseteq \cD + \cF_1$   where   $\cF_1$ denotes the set of all rank one operators on  $\hilb$.   
\end{eg}


\begin{rem} \label{rem4.10} 
As we have just seen, the projections $Q$ and $R$ obtained in Theorem~\ref{thm4.8} do not have to lie in $\cD$.   We remark, however, that if one of $Q$ and $R$ is zero -- say $R = 0$ -- then $Q$ must commute with every $D \in \cD$ appearing as the diagonal of an element $T \in \cA$.  That is, for all $T = D_T + F_T \in \cA$, $Q D_T = D_T Q$.

Indeed, choose $T_0 \in \cA$ and $P_0$ a projection in $\cD$ so that 
\[
\mathrm{rank} (P_0^\perp T_0 P_0) = \kappa = \sup_{T \in \cA,\  P  \in \cP(\cD) } \mathrm{rank}\, (P^\perp T P).\]
Since $T_0 = D_0 + F_0$ for some $D_0 \in \cD$ and $F_0 \in Q \bofh$ by Theorem~\ref{thm4.8}, it is clear that $\mathrm{rank}\, P_0^\perp F_0 P = \kappa$, and in particular, $\mathrm{rank}\, F_0 = \kappa = \mathrm{rank}\, Q$.   In other words, $\ran\, F_0 = Q\hilb$.   

For any $T = D + F \in \cA$, $T_0 T \in \cA$, and $T T_0 = (D D_0) + (F  D_0 + D F_0 + F F_0)$.   Since $D  D_0 \in \cD$ and $F D_0 + D F_0 + F F_0$ is finite-rank, the uniqueness of this decomposition (it is here that we need the measure to be continuous) implies that $F D_0 + D F_0 + F F_0 \in Q \bofh$.   Hence $Q (F D_0 + D F_0 + F F_0) = F D_0 + D F_0 + F F_0$, which implies that $Q D F_0 = D F_0$.   But $\ran\, F_0 = \ran\, Q$, and thus $Q D Q = D Q$.   Thus $\ran\, Q$ is a finite-dimensional invariant subspace for the normal operator $D$.   But any such subspace must be reducing for $D$, and hence $Q D = D Q$.
\end{rem}

\bigskip

The rest of this section is motivated by \cite[Proposition~5.1]{APTT} which states that if $T\in\cB(\cX)$ is such that every half-space in $\fX$ is $T$-almost-invariant then $T$ has many invariant subspaces. We show that, at least for the case of operators acting on a Hilbert space $\cH$, much more is true: the operator must be of the form $\alpha I+F$ where $\mathrm{rank}\,F<\infty$.

\begin{prop} \label{nov25_1.2}
Let $T \in \bofh$.   Suppose that given any masa $\cD$ for $\bofh$, we can find $D \in \cD$ and a finite-rank operator $F$ so that $T = D + F$.   Then there exists $\alpha \in \bbC$ and $H$ a finite-rank operator so that $T = \alpha I + H$.
\end{prop}
	
\begin{pf}
Let $0 \not = P$ be any projection in $\bofh$.   By a routine application of Zorn's Lemma, we can find a masa $\cD$ with $P \in \cD$.   Let $T = D + F$, where $D \in \cD$ and $F \in \cF(\hilb)$.   Then $D, P \in \cD$ implies that $D P = P D$, and hence $T P - P T = F P - P F \in \cF(\hilb)$.   That is, $T P - P T \in \cF(\hilb)$ for every projection $P$ in $\bofh$.

By a result of Fillmore~\cite{Fil1966} (see also the result of Matsumoto~\cite{Mat1984}), every $X \in \bofh$ is a finite linear combination of projections.  	It follows that $T X  - X T \in \cF(\hilb)$ for every $X \in \bofh$.   Let $(e_n)_{n=1}^\infty$ be an orthonormal basis for $\hilb$, and let $\cD_0$ denote the set of diagonal operators on $\hilb$ relative to this basis.   Since $\cD_0$ is again a masa in $\bofh$, we may write $T = D_0 + F_0$ for some $D_0 = \mathrm{diag}\, (d_n)_{n=1}^\infty \in \cD_0$ and $F_0 \in \cF(\hilb)$.   Let $S \in \bofh$ denote the unilateral forward shift operator relative to this basis, so that $S e_n = e_{n+1}$ for all $n \ge 1$.   Since $S T - T S \in \cF(\hilb)$, we also have $S D_0 - D_0 S \in \cF(\hilb)$.   But 
\[
S D_0 - D_0 S = \begin{bmatrix} 0 & & & & \\ d_1 - d_2 & 0 & & & \\ & d_2 - d_3 & 0 & & \\ & & d_3 - d_4 & 0 & \\ & & & \ddots & \ddots \\ \end{bmatrix}. \]
This operator is finite-rank if and only if there exists $N \ge 1$ so that $d_{n} - d_{n+1} = 0$ for all $n \ge N$; i.e. $d_n = d_N$ for all $n \ge N$.  Thus $D_0 \in \bbC I + \cF(\hilb)$, and hence $T = D_0 + F_0 \in \bbC I + \cF(\hilb)$ as well.
\end{pf}

\begin{cor}
Suppose that $T \in \bofh$ and that every half-space of $\hilb$ is almost-invariant for $T$.   Then $T = \alpha I + H$ for some $\alpha \in \bbC$ and $H \in \bofh$ a finite-rank operator.
\end{cor}		
		
\begin{pf}
By Theorem~\ref{thm4.5}, we know that if $\cD$ is a masa in $\hilb$ then $T = D + F$ for some $D \in \cD$ and some finite-rank operator $F$.  Since this is true for all masas, the result then follows from Proposition~\ref{nov25_1.2}.
\end{pf}



\section{Almost-reducing subspaces} \label{sec5}

As we have already seen, if $\cX$ is a Banach space and $\mathcal{A}$ is a norm-closed subalgebra of $\mathcal{B}(\cX)$ admitting a  complemented  almost-invariant half-space, then $\mathcal{A}$ admits an invariant half-space.  Let $\mathcal{H}$ be a Hilbert space and $\varnothing \not = \mathcal{S} \subseteq \mathcal{B}(\mathcal{H})$.  
We shall say that a half-space $\mathcal{M}$ of $\mathcal{H}$ is \term{reducing}  for $\mathcal{S}$ if the orthogonal projection $Q$ of $\mathcal{H}$ onto $\mathcal{M}$ lies in the commutant $\mathcal{S}^\prime := \{ T \in \mathcal{B}(\mathcal{H}): T S = S T \mbox{ for all } S \in \mathcal{S}\}$ of $\mathcal{S}$.   Equivalently, $\mathcal{M}$ must be  invariant for both $\mathcal{S}$ and for $\mathcal{S}^* := \{ S^*: S \in \mathcal{S}\}$.  We shall say that $\cM$ is \term{almost reducing} for $\cS$ if both $\cM$ and $\cM^\perp$ are almost-invariant for $\cS$, or equivalently, if $\cM$ is almost-invariant for both $\cS$ and $\cS^*$.

In light of the results of Section~\ref{sec3}, it is reasonable to ask whether, in the Hilbert space setting, a norm-closed subalgebra $\mathcal{A}$ of operators in $\mathcal{B}(\mathcal{H})$ admitting an {almost-reducing} half-space admits a reducing subspace.   The following example shows that this need not be the case.

\bigskip

\begin{eg} \label{eg5.1}
Let $\mathcal{H} = \ell_2(\bbZ)$ with orthonormal basis $\{ e_n \}_{n \in \bbZ}$.  Recall that for $x, y \in \mathcal{H}$, $x \otimes y^*$ represents the rank-one operator $x \otimes y^* (z) = \langle z, y \rangle x$ for all $z \in \mathcal{H}$.

Let $\mathcal{A}$ denote the algebra of operators $A$ whose matrix $[a_{i j}]$ relative to the basis $\{ e_n\}_{n \in \bbZ}$ satisfies:
\[
a_{i j} = 0 \mbox{ if } 
\begin{cases}
		| i - j| \ge 2, \mbox{ or } \\ 
		| i - j| = 1 \mbox{ and } i \mbox{ is odd.}
\end{cases} \]		
In other words, every matrix $A$ in $\mathcal{A}$ is of the following form:

\[
\bordermatrix{\ & \ldots & e_{-3} & e_{-2} & e_{-1} & e_0 & e_1 & e_2 & e_3 & \ldots \cr
\vdots & \ddots & & & & & & & & \cr
e_{-3}& * & * & * & & & & & & \cr
e_{-2}& & & * & & & & & & \cr
e_{-1}& & & * & * & * & & & & \cr
e_{0}& & & & & * & & & &  \cr
e_{1}& & & & & * & * & * & & \cr
e_{2}& & & & & & & * & & \cr
e_{3}& & & & & & & * & * & * \cr
\vdots & & & & & & & & & \ddots }
\]

\bigskip

Observe that if $\mathcal{D}$ denotes the diagonal (bounded) operators relative to $\{ e_n\}_{n \in \bbZ}$, then $\mathcal{D} \subseteq \mathcal{A}$, and $\mathcal{D}$ is a masa in $\mathcal{B}(\mathcal{H})$. If $\mathcal{M}$ is a reducing subspace for $\mathcal{A}$, and if $Q$ denotes the orthogonal projection of $\mathcal{H}$ onto $\mathcal{M}$, then $Q \in \mathcal{A}^\prime$, whence $Q \in \mathcal{D}^\prime = \mathcal{D}$.  Thus the only possible reducing subspaces must be standard  subspaces.  (Recall that a \term{standard subspace} relative to this masa is a subspace densely spanned by the basis vectors it contains.)

Since $Q \in \mathcal{D}$, we can write $Q = \mathrm{diag} \{ q_n\}_{n \in \bbZ}$.  Since $Q = Q^2$, we have that $q_n \in \{ 0, 1\}$ for all $n \in \bbZ$.   Suppose that $Q \not = 0$, so that there exists $m \in \bbZ$ for which $q_m = 1$.

\bigskip

\noindent{\textit{Case 1:}}\ \ 
If $m$ is even, then $W_m := e_{m-1} \otimes e_m^* \in \mathcal{A}$.   Since $Q \in \mathcal{A}^\prime$, $Q W_m = W_m Q$, which yields:
\[
q_{m-1} (e_{m-1} \otimes e_m^*) = q_m (e_{m-1} \otimes e_m^*). \]
Thus $q_{m-1} = q_m = 1$.

Similarly, $V_m := e_{m+1} \otimes e_m^* \in \cA$, and the equation $Q V_m = V_m Q$ implies
\[
q_{m+1} (e_{m+1} \otimes e_m^*) = q_m (e_{m+1} \otimes e_m^*). \]
Thus $q_{m+1} = q_m = 1$.

\smallskip
\noindent{\textit{Case 2:}}\ \ 
If $m$ is odd, we can employ a similar argument, only this time setting\linebreak $W_m = e_m \otimes e_{m+1}^*$ and $V_m = e_m \otimes e_{m-1}^*$.   Again, $V_m, W_m \in \mathcal{A}$ and so $Q V_m = V_m Q$, while $Q W_m = W_m Q$, which imply that $q_{m-1} = q_{m+1} = q_m = 1$.

From this it readily follows that $q_n = 1$ for all $n \in \bbZ$, and so $Q = I$.  
We have shown that $\mathcal{A}$ does not admit any proper reducing subspaces.

\bigskip

On the other hand, $\mathcal{A}$ admits a plethora of almost-reducing half-spaces.   For example, if for $n \in \bbZ$ we set $\mathcal{H}_n := \mathrm{span}\{ e_k: k \le n\}$, then each $\mathcal{H}_n$ is easily seen to be an almost-reducing half-space for $\mathcal{A}$ (with defect equal to 1).
\end{eg}

\begin{eg} \label{eg5.2}
With a bit more effort, it is possible to construct a single irreducible operator $T \in \mathcal{B} (\ell_2(\bbZ))$ such that $\mathcal{A}_T := \overline{\mathrm{alg}(T)}$ admits a large number of almost-reducing half-spaces.   We shall exhibit such an operator $T$ which lies in the algebra $\mathcal{A}$ defined above, implying that $\mathcal{A}_T \subseteq \mathcal{A}$.   In particular, every almost-reducing subspace of $\mathcal{A}$ is automatically almost reducing for $\mathcal{A}_T$.

\bigskip

With $\{ e_n\}_{n \in \bbZ}$ as above, consider the operator $T$ whose matrix is given by
\[
T = 
\bordermatrix{\ & \ldots & e_{-3} & e_{-2} & e_{-1} & e_0 & e_1 & e_2 & e_3 & \ldots \cr
\vdots & \ddots & & & & & & & & \cr
e_{-3}& * & d_{-3} & v_{-2} & & & & & & \cr
e_{-2}& & & d_{-2} & & & & & & \cr
e_{-1}& & & w_{-2} & d_{-1} & v_0 & & & & \cr
e_{0}& & & & & d_0 & & & &  \cr
e_{1}& & & & & w_0 & d_1 & v_2 & & \cr
e_{2}& & & & & & & d_2 & & \cr
e_{3}& & & & & & & w_2 & d_3 & * \cr
\vdots & & & & & & & & & \ddots }.
\]
The values of $\{ d_n\}_{n \in \bbZ}$, $\{w_{2n}\}_{n \in \bbZ}$ and $\{ v_{2n}\}_{n \in \bbZ}$ will be chosen in such a way as to guarantee that $T$ is irreducible.  We first decompose 
\[
\ell_2(\bbZ) = \mathcal{H}_o \oplus \mathcal{H}_e, \]
where $\hilb_o = [ \{e_{2 k+1} : k \in \mathbb{Z}\} ]$ and $\mathcal{H}_e := [ \{ e_{2k}: k \in \mathbb{Z} \}]$.  Relative to this decomposition, we may write 
\[
T = \begin{bmatrix} D_o & T_2 \\ 0 & D_e \end{bmatrix}, \]
where $D_0 = \mathrm{diag} \{ d_{2k+1} \}_{k \in \mathbb{Z}}$ and $D_e = \mathrm{diag} \{ d_{2k}\}_{ k \in \mathbb{Z}}$.  The precise form of $T_2$ need not concern us just yet.

Suppose that $P = \begin{bmatrix} P_o & Q \\ Q^* & P_e \end{bmatrix}$ is a projection commuting with $T$.  Then 
\[
\begin{bmatrix} D_o P_o + T_2 Q^* & D_o Q + T_2 P_e \\ D_e Q^* & D_e P_e \end{bmatrix} = \begin{bmatrix} P_o D_o & P_o T_2 + Q D_e \\ Q^* D_o & Q^* T_2 + P_e D_e \end{bmatrix}. \]

In particular, we have \[ D_e Q^* = Q^* D_o.  \]

\bigskip

Recall that if $\mathcal{H}$ and $\mathcal{K}$ are Hilbert spaces, $A \in \mathcal{B}(\mathcal{H})$, $B \in \mathcal{B}(\mathcal{K})$, then the \emph{Rosenblum operator} $\tau_{A, B}: \mathcal{B}(\mathcal{H}, \mathcal{K}) \to \mathcal{B}(\mathcal{H}, \mathcal{K})$ defined by $\tau_{A, B} (X) = A X - X B$ has spectrum contained in $\sigma(A) - \sigma(B) := \{ \alpha - \beta : \alpha \in \sigma(A), \beta \in \sigma(B)\}$ (see, e.g., \cite[Corollary~3.2]{Herr89}).  In particular, therefore, if $\sigma(A) \cap \sigma(B) = \varnothing$, then $\tau_{A, B}$ is invertible, and hence is injective.  As such, if we choose $D_o$ and $D_e$ so that $\sigma(D_o) \cap \sigma(D_e) = \varnothing$, then $\tau_{D_e, D_o}$ injective combined with the equation $D_e Q^* = Q^* D_o$ implies that $Q^* = 0$, and hence $Q = 0$.    

Furthermore, this will then force $D_o P_o = P_o D_o$ and $D_e P_e = P_e D_e$.   By considering adjoints, we obtain $D_o^* P_o = P_o D_o^*$ and $D_e^* P_e = P_e D_e^*$, so that $P_o$ belongs to the commutant $(W^*(D_o))^\prime$ of the von Neumann algebra $(W^*(D_o))$ generated by $D_o$ and similarly, $P_e \in (W^*(D_e))^\prime$.   If we choose all $\{ d_{2k+1}: k \in \bbZ \}$ and $\{ d_{2k}: k \in \mathbb{Z}\}$ to be distinct, then $(W^*(D_o)^\prime = W^*(D_o) = [ \{ e_{2k+1} \otimes e_{2k+1}^*: k \in \mathbb{Z} \}]$, while $W^*(D_e)^\prime = W^*(D_e) = [ \{ e_{2k} \otimes e_{2k}^*: k \in \mathbb{Z} \}]$, which are simply the diagonal masas of $\mathcal{H}_o$ and $\mathcal{H}_e$ respectively.

To satisfy both of the above conditions, it suffices to choose $\{ d_{2k} : k \in \mathbb{Z} \}$ a dense subset of $[\frac{3}{4}, 1]$ with $d_{2k} \not = d_{2j}$ unless $j = k$ and $\{ d_{2k + 1}: k \in \mathbb{Z} \}$ dense in $[\frac{1}{4}, \frac{1}{2}]$ with $d_{2k +1} \not = d_{2j +1}$ unless $j = k$.  In particular $\sigma(D_o) \cap \sigma (D_e) = [\frac{1}{4}, \frac{1}{2}] \cap [\frac{3}{4}, 1] = \varnothing$.

Combined with the previous observations, we find that $P_o$ and $P_e$ are diagonal operators relative to the bases $\{ e_{2k+1}: k \in \mathbb{Z} \}$ and $\{ e_{2k}: k \in \mathbb{Z}\}$ respectively, which shows that 
\[
P = P_o \oplus P_e \simeq \mathrm{diag} \{ ..., p_{-2}, p_{-1}, p_0, p_1, p_2, ... \} \]
lies in the masa $\cD = [ \{ e_n \otimes e_n^*: n \in \mathbb{Z}\}]$ of $\mathcal{B}(\ell_2(\mathbb{Z}))$.  There remains only to choose $\{ v_{2k}: k \in \mathbb{Z}\}$ and $\{ w_{2k}: k \in \mathbb{Z}\}$ so as to guarantee that $p_i = p_j$ for all $i, j \in \mathbb{Z}$.

Now 
\[
P T = \bordermatrix{\ & \ldots & e_{-3} & e_{-2} & e_{-1} & e_0 & e_1 & e_2 & e_3 & \ldots \cr
\vdots & \ddots & & & & & & & & \cr
e_{-3}& * & p_{-3} d_{-3} & p_{-3} v_{-2} & & & & & & \cr
e_{-2}& & & p_{-2} d_{-2} & & & & & & \cr
e_{-1}& & & p_{-1} w_{-2} & p_{-1} d_{-1} & p_{-1} v_0 & & & & \cr
e_{0}& & & & & p_0 d_0 & & & &  \cr
e_{1}& & & & & p_1 w_0 & p_1 d_1 & p_1 v_2 & & \cr
e_{2}& & & & & & & p_2 d_2 & & \cr
e_{3}& & & & & & & p_3 w_2 & p_3 d_3 & * \cr
\vdots & & & & & & & & & \ddots },
\]
while 
\[
T P = \bordermatrix{\ & \ldots & e_{-3} & e_{-2} & e_{-1} & e_0 & e_1 & e_2 & e_3 & \ldots \cr
\vdots & \ddots & & & & & & & & \cr
e_{-3}& * & p_{-3} d_{-3} & p_{-2} v_{-2} & & & & & & \cr
e_{-2}& & & p_{-2} d_{-2} & & & & & & \cr
e_{-1}& & & p_{-2} w_{-2} & p_{-1} d_{-1} & p_0 v_0 & & & & \cr
e_{0}& & & & & p_0 d_0 & & & &  \cr
e_{1}& & & & & p_0 w_0 & p_1 d_1 & p_2 v_2 & & \cr
e_{2}& & & & & & & p_2 d_2 & & \cr
e_{3}& & & & & & & p_2 w_2 & p_3 d_3 & * \cr
\vdots & & & & & & & & & \ddots }.
\]
But $P T = T P$, and so
\begin{align*}
p_{2k - 1} v_{2k} & = p_{2k} v_{2k} \\
p_{2k+1} w_{2k} &= p_{2k} w_{2k}, 
\end{align*}
for all $k \in \mathbb{Z}$.   As long as $w_{2k} \not = 0 \not = v_{2k}$ for all $k \in \mathbb{Z}$, we find that $p_{2k-1} = p_{2k} = p_{2k+1}$ for all $k$, from which it clearly follows that $p_i = p_j$ for all $i, j \in \mathbb{Z}$.   For example, choosing $v_{2k} = 1 = w_{2k}$ for all $k \in \mathbb{Z}$ will do.

With these (highly non-unique) choices of $d_k, v_{2k}$ and $w_{2k}$, we find that any projection $P$ commuting with $T$ must be scalar - i.e., $P = 0$ or $P = I$.  It follows that the corresponding $T$ is irreducible, as required.
\end{eg}

\end{document}

%% file: LaTeXdefs.tex
\usepackage{latexsym}
\usepackage{amsfonts}

\newenvironment{thm}{\subsection{}{\textbf {Theorem.}}\em}{}
\newenvironment{prop}{\subsection{}{\textbf {Proposition.}}\em}{}
\newenvironment{cor}{\subsection{}{\textbf {Corollary.}}\em}{}
\newenvironment{lem}{\subsection{}{\textbf {Lemma.}}\em}{}

\newenvironment{pf}{\noindent{\textbf {Proof.}}}
{\begin{flushright}\eop \end{flushright}\smallskip}

\newenvironment{eg}{\subsection{}{\textbf {Example.}}}{\smallskip}

\newenvironment{rem}{\subsection{}{\textbf {Remark.}}}{\smallskip}

\newcommand\fM{\ensuremath{\mathfrak M}}

\newcommand\fX{\ensuremath{\mathfrak X}}

\newcommand\cA{\ensuremath{\mathcal A}}
\newcommand\cB{\ensuremath{\mathcal B}}
\newcommand\cC{\ensuremath{\mathcal C}}
\newcommand\cD{\ensuremath{\mathcal D}}
\newcommand\cE{\ensuremath{\mathcal E}}
\newcommand\cF{\ensuremath{\mathcal F}}

\newcommand\cH{\ensuremath{\mathcal H}}

\newcommand\cK{\ensuremath{\mathcal K}}
\newcommand\cL{\ensuremath{\mathcal L}}
\newcommand\cM{\ensuremath{\mathcal M}}

\newcommand\cP{\ensuremath{\mathcal P}}

\newcommand\cS{\ensuremath{\mathcal S}}

\newcommand\cV{\ensuremath{\mathcal V}}
\newcommand\cW{\ensuremath{\mathcal W}}
\newcommand\cX{\ensuremath{\mathcal X}}
\newcommand\cY{\ensuremath{\mathcal Y}}
\newcommand\cZ{\ensuremath{\mathcal Z}}

\newcommand\bbC{\ensuremath{\mathbb C}}

\newcommand\bbN{\ensuremath{\mathbb N}}

\newcommand\bbZ{\ensuremath{\mathbb Z}}

\newcommand\hilb{\ensuremath{\mathcal H}}

\newcommand\eop{{{\hfil \ensuremath \Box}}}

\newcommand\norm{\ensuremath {\Vert}}

\newcommand\bofh{\ensuremath{\cB ( \cH)}}

\newcommand\ran{\ensuremath {\mathrm {ran}}}

\newcommand\nul{\ensuremath {\mathrm {nul}}}

